\documentclass{article}
\usepackage[margin=1.25in]{geometry}

\usepackage{amsmath}
\usepackage{amsthm}
\usepackage{amsfonts}
\usepackage[latin1]{inputenc}
\usepackage[all]{xy}
\usepackage{amssymb}

\usepackage{tikz}
\usetikzlibrary{matrix}
\usepackage{tikz-cd}
\usetikzlibrary{cd}

\theoremstyle{definition}\newtheorem{teo}{Theorem}[section]
\theoremstyle{definition}\newtheorem{coro}[teo]{Corollary}
\theoremstyle{definition}\newtheorem{defi}[teo]{Definition}
\theoremstyle{definition}\newtheorem{lemma}[teo]{Lemma}
\theoremstyle{definition}\newtheorem{prop}[teo]{Proposition}
\theoremstyle{remark} 
\theoremstyle{remark}\newtheorem{strc}[teo]{}
\theoremstyle{remark}
\theoremstyle{remark}

\newcommand{\pch}{pch(\mathbb{X})}
\newcommand{\chn}{ch(\mathbb{X})}

\newcommand{\X}{\mathbb{X}}

\newcommand{\A}{\mathbb{A}}
\newcommand{\Z}{\mathbb{Z}}

\newcommand{\T}{\mathcal{T}}
\newcommand{\F}{\mathcal{F}}

\usepackage{fancyhdr}
\fancypagestyle{firststyle}
{
   \fancyhf{}
   \fancyfoot[L]{\footnotesize \today}
   
}

\title{Torsion theories of simplicial groups with truncated Moore complex}
\author{Guillermo Andr\'es L\'opez Cafaggi}
\date{}

\begin{document}
\maketitle
\thispagestyle{firststyle}

\begin{abstract}
We introduce a linearly ordered lattice $\mu(Grp)$ of torsion theories in simplicial groups.  The torsion theories are defined where the torsion/torsion-free subcategories are given by the simplicial groups with bounded above/below Moore complex, respectively. These torsion theories extend naturally the torsion theories in internal groupoids in groups. Connections of this lattice with the homotopy groups are established since the homotopy groups of a simplicial group can be calculated as the quotients of torsion subojects.
\end{abstract}

\section{Introduction}

The notion of semi-abelian category \cite{JMT02} allows a categorical and unified treatment of the categories of groups, rings, Lie-algebras and other non-abelian categories in a similar way as abelian categories generalise abelian groups and categories of modules. Torsion theories were originally introduced for abelian categories by Dickson, and have been generalized by several authors to different non-abelian categories as for example in \cite{BouGrn06}, \cite{CDT06} and \cite{JT07}.

For a torsion theory in a semi-abelian category $\X$ we mean a pair $(\T, \F)$ of full subcategories  such that:
\begin{enumerate}
\item any morphism $f: T \to F$ with $T$ in $\T$ and $F$ in $\F$ is the zero morphism;
\item for any object $X$ in $\X$ there is a short exact sequence
\[\begin{tikzcd} 0\ar[r] & T_X\ar[r] & X\ar[r] & F_X\ar[r] &0 \end{tikzcd}\]
with $T_X$ in $\T$ and $F_X$ in $\F$.
\end{enumerate}
An internal groupoid $X$ in $\X$ is a diagram:
\[\begin{tikzcd}  X_2\ar[r, shift right=3.5, "p_0"]\ar[r, "p_1"]\ar[r, shift left=3.5, "p_2"] & X_1\ar[r, shift right=3.5, "d_1"]\ar[r, shift left=3.5,"d_1"] &  X_0\ar[l, "s_0"'] \end{tikzcd}\]
where 
\[\begin{tikzcd} X_2\ar[r, "p_1"]\ar[d, "p_0"] & X_1\ar[d, "d_0"] \\ X_1\ar[r, "d_1"] & X_0\end{tikzcd}\] is a pullback square. The objects $X_0$ and $X_1$ are called the `objects of objects' and the `object of arrows' of $X$, the morphisms $d_0$, $d_1$ are called `domain' and `codomain' and the morphisms $p_0,p_1,p_2,d_0,d_0,s_0$ satisfy the usual equations that determine a category.
The category $Grpd(\X)$ of internal groupoids in a semi-abelian category $\X$, which is itself semi-abelian, exhibits two examples of non-abelian torsion theories. The first is given by the pair $(Ab(\X), Eq(\X))$ where $Ab(\X)$ is the category of internal abelian obejcts in $\X$ and $Eq(\X)$ is the category of equivalence relations, i.e. internal groupoids where the induced morphism $(d_0,d_1):X_1 \to X_0^2$ is monic \cite{BouGrn06}. The second example is given by $(Conn(Grpd(\X)), Dis(\X))$, where $Conn(Grpd(\X))$ is the category of connected internal groupoids and $Dis(\X)$ is the category of discrete groupoids. Since for an internal groupoid $X$ the nerve $\mathcal{N}(X)$ is a simplicial object in $\X$ it is natural to ask if there are torsion theories in simplicial objects such that they expand or generalize those of internal groupoids.
 
In section 2, we recall the basics theory of torsion theories in semi-abelian categories. Section 3 and 4 introduce two different families of torsion theories, $\mathcal{COK}_n$ and $\mathcal{KER}_n$, in the category of proper chain complexes and we exhibits some connections with the homological aspects of chain complexes. Section 5 and 6 introduce the torsion theories $\mu_{n \geq}$, $\mu_{\geq n}$ in simplicial groups whose associated Moore complex behave as those in proper chains. Section 7 studies the homotopy groups of the simplicial groups defined by the torsion theories of the lattice $\mu(Grp)$. In particular, the homotopy groups of a simplicial group $X$ can be studied using torsion subobjects. 

\section{Torsion theories in semi-abelian categories}

\begin{strc} \textit{Notation}. By a regular category $\X$ we mean a finitely complete category with coequalizers of kernel pairs with the property that any morphism $f:X \to Y$ in $\X$  factors as a regular epimorphism $e_f$ followed by a monomorphism $m_f$:
\[\begin{tikzcd} X\ar[rr, "f"]\ar[rd, "e_f"'] & & Y \\ & f(X)\ar[ru, "m_f"'] & \end{tikzcd} \]
and these factorizations are pullback stable. As usual, we will call the subobject represented by $m_f$ the \textit{image} of $f$.
A category $\X$ is pointed if it has a zero object $0$, i.e. an object which is both initial and terminal. For any pair of objects $X, Y$ in $\X$ the unique morphism $X \to Y$ that factors through the zero object will be denoted by $0$.

A regular category $\X$ is called (Barr-)exact if any equivalence relation is a kernel pair $Eq(f)$ for some morphism $f$ in $\X$ \cite{Bar71}.
\end{strc}

\begin{defi} \cite{JMT02} A category $\X$ is called \textit{semi-abelian} if it is pointed, (Barr-)exact, protomodular in the sense of Bourn (\cite{Bou91}) and has binary coproducts.
\end{defi}

In a semi-abelian category, a short exact sequence is a pair of composable morphisms $(k,p)$, as in the diagram
\[\begin{tikzcd} 0\ar[r] & K\ar[r, "k"] & X\ar[r, "p"] & Y\ar[r] &0\end{tikzcd} \]
such that $k=ker(p)$ is the kernel of $p$ and $p=cok(k)$ is the cokernel of $k$. In such a short exact sequence the object $Y$ will be denoted as $X/K$. Recall that in a semi-abelian category  $\X$ regular epimorphisms are normal epimorphisms, that is cokernels of some morphisms in $\X$. 

We will need the following results. 

\begin{lemma}\label{n3t}  \cite{BorBou04} Let $\X$ be a semi-abelian category. Given two normal subobjects $k: K \to A$ and $l: L \to A$ such that $k \leq l$, i.e. $k$ factors through $l$, then there is a short exact sequence:
\[\begin{tikzcd} 0\ar[r]& L/K\ar[r]&  A/K\ar[r] & A/L\ar[r]& 0 \end{tikzcd}. \]
\end{lemma}

\begin{prop}\label{iddet} \cite{JMT02} Let $\X$ be a semi-abelian category then it satisfies the following property: Given a commutative diagram in $\X$:
\[\begin{tikzcd}  A\ar[r, "m"]\ar[d, "p"] & B\ar[d, "q"]  \\  C\ar[r, "n"] & D \end{tikzcd} \]
with $p$ and $q$ normal epimorphisms, $m$ a normal monomorphism, and $n$ a monomorphism; then $n$ is a normal monomorphism.
\end{prop}

 Torsion theories can be defined in a more general context, but in this article we will restrict to semi-abelian categories.

\begin{defi} Let $\X$ be a semi-abelian category. A \textit{torsion theory} in $\X$ is a pair $(\T, \F)$ of full and replete subcategories of $\X$ such that:
\begin{enumerate}
\item[TT1] A morphism $F:T \to F$ with $T$ in $\T$ and $F$ in $\F$ is a zero morphism.
\item[TT2] For any object $X$ in $\X$ there is a short exact sequence:
\begin{equation}\label{sestt}
\begin{tikzcd} 0\ar[r] &T_X\ar[r, "\epsilon_X"] & X \ar[r, "\eta_X"] & F_X \ar[r]& 0\end{tikzcd} 
\end{equation} 
with $T_X$ in $\T$ and $F_X$ in $\F$ (which is necessarily unique up to isomorphism). 
\end{enumerate}
\end{defi}

In a torsion theory $(\T, \F)$, $\T$ is the torsion category whose objects are called torsion objects, and similarly $\F$ is the torsion-free category of the torsion theory. Torsion subcategories are normal mono-coreflective subcategories of $\X$, i.e. coreflective subcategories such that each component $\epsilon_X$ of the counit $\epsilon$ is a normal monomorphism, while torsion-free subcategories are normal epi-reflective subcategories, so that each component $\eta_X$ is a normal epimorphism:
\begin{equation}\label{ttadj}
\begin{tikzcd}[column sep=large] \T \ar[r, bend left, "J"]\ar[r, phantom, "\perp"] & \X \ar[r, bend left, "F"]\ar[l, bend left,"T"]\ar[r, phantom, "\perp"] & \F \ar[l, bend left, "I"] \end{tikzcd}.
\end{equation}

The $X$-component of the counit $\epsilon$ of $J \dashv  T$ and of the unit $\eta$ of $F \dashv I$ both appear in the short exact sequence (\ref{sestt}). A subcategory $\A$ of $\X$ is closed under extensions in $\X$ if every time we have a short exact sequence
\[\begin{tikzcd} 0\ar[r] & A \ar[r] & X \ar[r] & B \ar[r]& 0\end{tikzcd} \]
with $A$ and $B$ in $\A$ then  $X$ belongs to $\A$. In a torsion theory both $\T$ and $\F$ are closed under extensions in $\X$ \cite{BouGrn06}.

\begin{defi} Let $\X$ be a semi-abelian category. A \textit{preradical} in $\X$  is a normal subfunctor $\sigma:r \to Id_\X$ of the identity functor of $\X$, i.e. for all object $X$ we have a normal monomorphism $\sigma_X: r(X) \to X$ and for every morphism $f:X \to Y$ a commutative diagram:
\begin{equation}\label{prds}
\begin{tikzcd} X\ar[r, "f"] & Y \\ r(X)\ar[r, "r(f)"]\ar[u, "\sigma_X"] & r(Y)\, .\ar[u, "\sigma_Y"] \end{tikzcd}
\end{equation}
Moreover, a preradical $r$ is called:
\begin{itemize}
\item \textit{idempotent} if $rr(X) = r(X)$ for all objects $X$;
\item a \textit{radical} if $r(X/r(X))=0$ for all objects $X$;
\item \textit{hereditary} if for every monomorphism $f:X \to Y$, diagram (\ref{prds}) is a pullback.
\end{itemize}
\end{defi}

Given a preradical $r$ we can consider the $r$-torsion subcategory $\T_r$ and the $r$-torsion-free subcategory $\F_r$ of $\X$:
\[\T_r=\{X \in \X \mid r(X) \cong X\} \quad \mbox{and} \quad \F_r=\{X \in \X \mid r(X) \cong 0\}. \]
In general, the pair $(T_r, F_r)$ only satisfies axiom TT1 of a torsion theory. Conversely, a torsion theory $(\T, \F)$ defines an idempontent radical 
\[t= J T: \begin{tikzcd} \X\ar[r] & \X \end{tikzcd}.\]
In fact, there is a bijection:
\[\{\mbox{torsion theories in $\X$} \} \longleftrightarrow  \{\mbox{idempotent radicals in $\X$}\}. \]
It is easy to see that a hereditary preradical is always idempotent. Conversely, an idempotent preradical is hereditary if and only if the category $\T_r$ is closed under subobjects in $\X$, i.e. for every monomophism $m: X\to Y$ with $Y$ in $\T_r$ then $X$ is in $T_r$. Futhermore, the previous bijection is restricted to a bijection:
\[\{\mbox{hereditary torsion theories in $\X$} \} \longleftrightarrow  \{\mbox{hereditary radicals in $\X$}\}. \]

A torsion theory $(\T, \F)$ is called \textit{hereditary} if $\T$ is closed under subobjects. Similarly, $(\T, \F)$ is called \textit{cohereditary} if $\F$ is closed under quotients, i.e. for every normal epimorphism $p: X \to Y$ with $X$ in $\F$ then so is $Y$ in $\F$ (see \cite{CDT06}). It is also useful to recall that in any torsion theory $(\T, \F)$, $\T$ is always closed under quotients and $\F$ is closed under subobjects.

 In order to characterize torsion-free subcategories among normal epi-reflective subcategories, it is useful to recall the following result, which was first proved in \cite{BouGrn06} for homological categories: 

\begin{teo}\label{tfnormal} (\cite{BouGrn06}, \cite{EveGrn13}, \cite{CGJ18}) Let $\X$ be a semi-abelian category and $F \dashv I: \X \to \F$ a normal epi-reflective subcategory of $\X$ with unit $\eta$, then the following are equivalent:
\begin{enumerate}
\item $\F$ is a torsion-free subcategory of $\X$;
\item the induced radical of $F \dashv I$ is idempotent;
\item the reflector $F: \X \to \A$ is semi-left-exact;
\item the feflector $F: \X \to \A$ is normal, i.e. $F(ker(\eta_X))=0$ for every object $X$ in $\X$.
\end{enumerate}
Under these conditions the corresponding torsion category of $\F$ is given by the full subcategory $\T= Ker(F)=\{X \mid F(X) \cong 0\}$. So, $(Ker(F), \F)$ is a torsion theory in $\X$.
\end{teo}
\begin{proof} Equivalences $(1)\Leftrightarrow (2) \Leftrightarrow (3)$ are proved in  \cite{EveGrn13} and $(3)\Leftrightarrow (4)$ is proved in \cite{BouGrn06}. 
\end{proof}

\begin{strc}
Given torsion theories $(\T,\F)$ and $(\mathcal{S}, \mathcal{G})$ in $\X$ with associated idempotent radicals $\tau$ and $\sigma$ we have that $\T \subseteq \mathcal{S}$ if and only if $\mathcal{G} \subseteq \F$, this allows  us to define a partial order in the (possibly big) lattice $\X tors$ of torsion theories in $\X$:
\[ (\T,\F) \leq (\mathcal{S} , \mathcal{G}) \quad \mbox{if and only if} \quad \T \subseteq \mathcal{S}.\]
In this case we have that $\tau \leq \sigma$, so for an object $X$ in $\X$ we have $\tau(X) \leq \sigma(X)$. The lattice $\X tors$ has as  bottom and  top element the trivial torsion theories denoted as:
\[0:=(0, \X)\quad \mbox{and} \quad  \X:= (\X,0). \]

Given preradicals $\tau \leq \sigma$ we can define the quotient endofunctor as:
\[\sigma/\tau : \begin{tikzcd} \X\ar[r] & \X \end{tikzcd}, \quad \sigma/ \tau(X)=\sigma(X)/\tau(X). \]
and as consequence of Lemma \ref{n3t} for each object $X$ we have a short exact sequence:
\begin{equation}\label{eq3it}
\begin{tikzcd} 0\ar[r]& \sigma(X)/\tau(X)\ar[r]& X/\tau(X)\ar[r] & X /\sigma(X)\ar[r] &0\end{tikzcd}.
\end{equation}
\end{strc}

For abelian categories the next result is due to P. Gabriel (\cite{Sten}). A localization of a category $\X$ is a reflective subcategory $ L \dashv I :\X \to\A$ such that $L$ preserves finite limits.

\begin{teo}\label{loctt} Let $ L \dashv I: \X \to\A$ be a localization of a semi-abelian category $\X$ with unit $\eta$. The subcategories of $\X$
 \[\T_L=ker(L)=\{X\mid L(X)\cong 0\} = \{X \mid \eta_X=0 \}\] 
 and 
 \[\F_L=\{X\mid \eta_X:X\to IL(X)\ \mbox{is monic} \}\] define a torsion theory $(\T_L, \F_L)$ in $\X$.
\end{teo}
\begin{proof} 

First notice that for any object $X$ the morphism $X \to 0$ factors through $\eta_X$. Then, if $\eta_X=0$ the morphism $\eta_X$ factors through $X \to 0$ and, so $L(X) \cong 0$ since $L$ is a reflection. Hence, for any object $X$ we have that $L(X)=0$  if and only if $\eta_X=0$. 
TT1) For a morphism $f:X \to Y$ consider the diagram given by the naturality of $\eta$:
\[\begin{tikzcd} X\ar[d, "\eta_X"']\ar[r, "f"] & Y\ar[d, "\eta_Y"] \\ IL(X)\ar[r, "IL(f)"'] & IL(Y)\, .  \end{tikzcd}\]
Now, if $\eta_X=0$ and $\eta_Y$ is monic it is clear that $f$ is the zero morphism.

TT2) Consider for an object $X$ the normal epi-mono factorization $(p,m)$ of $\eta_X$ and the short exact sequence
\begin{equation}\label{sesloc}
\begin{tikzcd} 0\ar[r] & ker(\eta_X) \ar[r, "k"] & X \ar[r, "p"] \ar[rd, "\eta_X"'] & \eta_X(X)\ar[d, "m"]\ar[r] &0 \\ && & IL(X)\, .& \end{tikzcd} 
\end{equation}
To see that $ker(\eta_X)$ is torsion, consider the commutative diagram:
\[\begin{tikzcd} ker(\eta_X) \ar[d, "\eta_{ker(\eta_X)}"']\ar[r, "k"] & X \ar[d, "\eta_X"] \\ IL(ker(\eta_X)) \ar[r, "IL(k)"'] & IL(X) \, .  \end{tikzcd}\]
Since $L$ preseves finite limits then $IL(k)$ is monic and since $\eta_X k=0$ this implies that $\eta_{ker(\eta_X)}=0$. To see that $\eta_X(X)$ is torsion-free consider the diagram:
\[\begin{tikzcd} 
X\ar[rr, "\eta_X"]\ar[rd, "p"] \ar[dd, "\eta_X"] &  & IL(X)\ar[dd, "\eta_{IL(X)}"]
\\ & \eta_X(X)\ar[ru, "m"] \ar[dd, "\eta_{\eta_X(X)}" near start] & 
\\IL(X)\ar[rr, "IL(\eta_X)" near start, crossing over]\ar[rd,"IL(p)"'] & & ILIL(X)
\\ & IL(\eta_X(X)) \, .\ar[ru, "IL(m)"'] & \end{tikzcd}\]
Notice that since $\A$ is a reflective subcategory then $IL(\eta_X)$ and $\eta_{IL(X)}$ are isomorphisms. Finally, $\eta_{\eta_X(X)}$ is also a monomorphism.
\end{proof}

It is also worth mentioning that, under the assumptions from above, a localization $L: \X \to \A$ induces a preradical on $\X$ as $r=ker(\eta)$, so we have $\T_L= \T_r$.

\begin{coro} The torsion theory $(\T_L, \F_L)$ induced by a localization $L \dashv i$ of a semi-abelian category $\X$ is hereditary.
\end{coro}
\begin{proof} Since $L$  preserves finite limits it preserves monomorphisms so if $m: S \to X$ with $L(X)=0$ then $L(S)=0$. 
\end{proof}

\begin{lemma} Let $(\T, \F)$ be a hereditary torsion theory in a semi-abelian category $\X$. Then $\T$ is closed under finite limits in $\X$. In particular, $\T$ is closed under kernel pairs of morphisms in $\T$.
\end{lemma}
\begin{proof}
Since $\T$ is closed under kernels of arrows in $\T,$ we only need to prove that $\T$ is closed under pullbacks of morphisms in $\T$. Consider the commutative diagram
\[\begin{tikzcd} ker(p_1)\ar[r]\ar[d, "p_0'"] & P \ar[r, "p_1"]\ar[d, "p_0"] & C\ar[d, "g"] \\ker(f)\ar[r] & A\ar[r, "f"] & B\, . \end{tikzcd} \] 
with $f$ and $g$ morphisms in $\T$. Since $P$ is a pullback then $ker(p_1)\cong ker(f)$. Now, since $\T$ is closed under subobjects, if $A$ is torsion then $ker(f)$ is also torsion. Consider $e,m$ the normal epi/mono factorization of $p_1$ and the short exact sequence:
\[\begin{tikzcd} 0\ar[r]&ker(p_1)\ar[r] & P\ar[r, "e"]& p_1(P)\ar[r]&0 \end{tikzcd}.\]
Then $p_1(P)$ is a subobject of $C$ so it is torsion, and finally, since $\T$ is closed under extension $P$ is torsion.
\end{proof}

A \textit{quasi-hereditary} torsion theory $(\T, \F)$ in $\X$ is a torsion theory such that $\T$ is closed under regular subobjects, i.e. if $e: X \to T$ is an equalizer with $T$ torsion then $X$ is torsion. In \cite{GrnRss07} quasi-hereditary torsion theories are studied in homological categories.

\begin{teo}\label{grnrss} \cite{GrnRss07} Let $(\T, \F)$ be a torsion theory in a homological category $\X$. The following are equivalent:
\begin{enumerate}
\item $(\T,\F)$ is quasi-hereditary.
\item The associated idempotent radical $t: \X  \to \X$ preserves finite limits.
\item The associated idempotent radical $t: \X \to \X$ preserves equalizers.
\item For every regular subobject $e: E \to A$ in $\X$ then $F(e)$ is a monomorphism in $\F$. 
\end{enumerate}
\end{teo} 

\begin{lemma}\label{tthexac} Let $(\T, \F)$ be a torsion theory in a semi-abelian category $\X$. Then $\T$ is an exact category.
\end{lemma}
\begin{proof}
We will first prove that an arrow $q$ in $\T$ is a regular epimorphism in $\X$ if and only if it is a regular epimorphism in $\T$. Clearly, if $q$ is a regular epimorphism in $\T$ and the inclusion $J: \T \to \X$ preserves colimits then $q$ is a regular epimorphism in $\X$. Now, if $q$ is a regular epimorphism in $\X$ it is a coequalizer of its kernel pair $K[q]$ in $\X$, and since $\T$ is closed under kernel pairs in $\X$ we have the isomorphism $Eq(q)\cong t(Eq(q))$, so $q$ is a regular epimorphism in $\T$. 

To prove pullback stability  of regular epimorphism consider the pullback diagram in $\T$:
\[\begin{tikzcd} P\ar[d, "p'"']\ar[r] & X\ar[d,"p"] \\ A\ar[r] & B  \end{tikzcd} \]
with $p$ a regular epimorphism in $\T$. Since the inclusion $i:\T \to \X$ preserves pullbacks and quotients, we have that $p$ is a regular epimorphism in $\X$, and so is $p'$ in $\T$. 

Finally, since $\T$ is closed under quotients and $\X$ is an exact category, any equivalence relation in $\T$ must be effective and $\T$ is an exact category.  
\end{proof}

\begin{teo}\label{hersemiab} Let $(\T, \F)$ be a hereditary torsion theory in a semi-abelian category $\X$. Then $\T$ is a semi-abelian category.
\end{teo}
\begin{proof} First since $\T$ is coreflective in $\X$ it is complete and cocomplete as well as pointed. By Lemma \ref{tthexac} $\T$ is exact. Finally, by Theorem \ref{grnrss} the full inclusion $J: \T \to \X$ preserve finite limits and hence short split exact sequences, so if $\X$ is protomodular then so is $\T$.   
\end{proof}

This theorem admits a dual version. A Birkhoff subcategory  $\mathbb{A}$ of a regular category $\X$, is a full regular epi-reflective subcategory that is closed under subobjects and quotients in $\X$. It is known that if $\X$ semi-abelian then so is any Birkhoff subcategory $\mathbb{A}$. 

\begin{coro}\label{cohersemi} Let $(\T, \F)$ be a cohereditary torsion theory in a semi-abelian category $\X$. Then $\F$ is a semi-abelian category.
\end{coro}
\begin{proof}
If $(\T, \F)$ is cohereditary, $\F$ is closed under quotients in $\X$, so $\F$ is a Birkhoff subcategory of $\X$.
\end{proof}

\section{Torsion theories in chain complexes}

Throughout this section $\X$ will denote a semi-abelian category.

A chain complex $M$ in $\X$ is a family of morphisms $\{\delta_i: M_i \to M_{i-1}\}_{i \in \mathbb{Z}}$ with the condition $\delta_i\delta_{i+1}=0$ for all $i$. A morphism of chain complexes $f: M \to N$ is a family of morphisms $f_i:M_i \to N_i$ such that $f_{i-1}\delta_i=\delta_if_{i}$ for all $i$. For a chain complex $M$ and for each $i$ we will write $e_i$ and $m_i$ for the normal epi/mono factorization of each $\delta_i$:
\[\begin{tikzcd}  M_i\ar[rr, "\delta_i"]\ar[rd, "e_i"] & & M_{i-1}\\ & \delta_i(M_i)\ar[ru, "m_i"] & \end{tikzcd}\]
And we call a chain complex $M$ \textit{proper} if each $\delta_i$ is a proper morphism i.e. each $m_i$ is a normal monomorphism for each $i$.
We will write $\chn$ for the category of chain complexes in $\X$ and $\pch$ for the subcategory of proper chain complexes. In \cite{EveLin04} it is noticed that since $\X$ is  a semi-abelian category then $\chn$ is also semi-abelian, but this is not the case for $\pch$, since it may not have kernels. However, $\pch$ does have cokernels as follows.

\begin{lemma} The category $\pch$ has cokernels and they are computed as in $\chn$.
\end{lemma}
\begin{proof}  For a morphism $f:M\to N$ of proper chain complexes $(M,d)$, $(N,\delta)$ consider the commutative diagram for each $i$
\[\begin{tikzcd}
M_i  \ar[rr,"f_i"]\ar[dd, "d_i"]
 & 
  & N_i\ar[rr, "p"]\ar[rd,"e_i"]\ar[dd, "\delta_i"]
   &
    & cok(f_i)\ar[dd, "\delta_i'",near start] \ar[rd, "e_i'"] &
\\ & & & \delta_i(N_i) \ar[ld, "m_i"] \ar[rr, "q'",near start,  crossing over] & & \delta_i'(cok(f_i))\ar[dl, "m_i'"]
\\
M_{i-1}\ar[rr, "f_{i-1}"]
 & 
  & N_{i-1} \ar[rr, "q"]
   &
    & cok(f_{i-1}) &
\end{tikzcd} \]
where $\delta_i'$ is induced by universal property of the cokernel $p$ and $m_i$, $e_i$ and $m_i'$, $e_i'$ are the normal epi-mono image factorizations of $\delta_i$ and $\delta_i'$ respectively. Now, since taking images is functorial we have $q'$ such that $q'e_i=e_i'p$ and $m_i'q'=qm_i$, then $q'$ is a normal epimorphism since $p$ and $e_i'$ are also normal epimorphisms. Finally, by Proposition \ref{iddet} if $m_i$ is a normal monomorphism  and $m_i'$ is a monomorphism, then $m_i'$ is a normal monomorphism. So, $\delta_i'$ is a proper morphism and it is the cokernel of $f$ in $\pch$.
\end{proof}

By a short exact sequence in $\pch$ we mean a short exact sequence in $\chn$ such that every object is a proper chain complex. We will introduce well-known functors in very different settings of algebraic topology that still makes sense in our context, we will follow the terminology of \cite{BrHiSi10}.

\begin{strc}\label{adjun} Let $\chn_{n \geq}$ be the category of $n$-truncated (above) chain complexes, i.e. chain complexes defined for degrees $n \geq i$ for a fixed $n\in \mathbb{Z}$. We can identify $\chn_{n \geq}$ with the full subcategory of $\chn$ of chain complexes with $M_i=0$ for $i>n$. Actually, we have the functors:
\begin{itemize}
\item  $\mathbf{tr}_n:\begin{tikzcd} \chn\ar[r] & \chn_{n \geq }\end{tikzcd}$ is the canonical (above) truncation:
 \[ \mathbf{tr}_n(M) = \begin{tikzcd}
M_n\ar[r] & M_{n-1}\ar[r] & M_{n-2}\ar[r] & \hdots
\end{tikzcd} \]
\item $\mathbf{sk}_n:\begin{tikzcd} \chn_{n \geq } \ar[r] & \chn\end{tikzcd}$ is the canonical inclusion or skeleton functor:
\[ \mathbf{sk}_n(M)= \begin{tikzcd}
\hdots\ar[r] &0\ar[r] & 0\ar[r] & M_n\ar[r] & M_{n-1}\ar[r] & \hdots
\end{tikzcd} \]
\item $\mathbf{cosk}_n:\begin{tikzcd} \chn_{n \geq }\ar[r] & \chn \end{tikzcd}$  the coskeleton functor is given by:
\[\mathbf{cosk}_n(M)= \begin{tikzcd}
\hdots\ar[r] & 0\ar[r] & Ker(\delta_n) \ar[r, "ker(\delta_n)"] & M_n\ar[r, "\delta_n"] & M_{n-1}\ar[r] &\hdots
\end{tikzcd} \]
\item  $\mathbf{cot}_n: \begin{tikzcd} \chn \ar[r] & \chn_{n \geq } \end{tikzcd}$ the (above) cotruncation functor:
\[\mathbf{cot}_n(M) =\begin{tikzcd}
Cok(\delta_{n+1})\ar[r, "\delta_n'"] & M_{n-1}\ar[r] & M_{n-2}\ar[r] & \hdots
\end{tikzcd}, \]
where $\delta'_n$ is induced by $\delta_n: M_n \to M_{n-1}$ and the universal property of $cok(\delta_{n+1})$.
\end{itemize}
This functors give a string of adjunctions:
\[ \mathbf{cot}_n \dashv \mathbf{sk}_n \dashv \mathbf{tr}_n \dashv \mathbf{cosk}_n:  
\begin{tikzcd}[row sep=large] 
\chn 
\ar[d, bend left, shift right=.5] \ar[d, bend right, shift right=4] \ar[d, phantom, "\dashv"] \ar[d,phantom,"\dashv", shift right=5.5 ] \ar[d,phantom, "\dashv", shift left=5.5 ] \\
\chn_{n \geq }
\ar[u, bend left, shift right=.5] \ar[u, bend right, shift right=4]
\end{tikzcd} .\] 
We will write $\mathbf{Sk}_n=\mathbf{sk}_n\mathbf{tr}_n$, $\mathbf{Cosk}_n=\mathbf{cosk}_n\mathbf{tr}_n$ and $\mathbf{Cot}_n=\mathbf{sk}_n\mathbf{cot}_n$.
\end{strc}

Notice that for abelian categories the cotruncation functor above is exactly L. Illusie's truncation functor in \cite{ill71}.

\begin{lemma} The category $\chn_{n-1 \geq}$ is a normal epireflective subcategory of $\chn$ with the adjunction $\mathbf{cot}_{n-1}\dashv \mathbf{sk}_{n-1}$. Moreover, the category $\chn_{n-1 \geq}$ is closed under extensions in $\chn$.  
\end{lemma}
\begin{proof}
For a chain complex $M$ the unit $\eta_M$ of $\mathbf{cot}_{n-1} \dashv \mathbf{sk}_{n-1}$  is given by:
\[\begin{tikzcd} M= \ar[d, "\eta_M"] & \dots\ar[r]  & M_n\ar[r, "\delta_n"]\ar[d] & M_{n-1}\ar[r, "\delta_{n-1}"]\ar[d,"cok(\delta_n)"]& M_{n-2}\ar[r, "\delta_{n-2}"]\ar[d, "1"] & \dots
\\ 
\mathbf{Cot}_{n-1}(M)=&\dots\ar[r] & 0\ar[r] & Cok(\delta_n)\ar[r, "\delta'_{n-1}"] & M_{n-2}\ar[r, "\delta_{n-2}"] & \dots \end{tikzcd} \]
which is a component-wise normal epimorphism, and hence, $\eta_M$ is a normal epimorphism in $\chn$. 

Since a short exact sequence in $\chn$ is a component-wise short exact sequence  then it is clear that $\chn_{n-1 \geq}$ is closed under extension in $\chn$.
\end{proof}

\begin{strc}
In \cite{JT07} conditions are given for a normal epireflective subcategory closed under extensions to be a torsion-free category. Here, $\chn_{n-1 \geq}$ provides a counter-example of this situation, in the sense that $\chn_{n-1\geq}$ is a normal epireflective subcategory closed under extensions of $\chn$ that is not a torsion-free subcategory. Indeed, the functor $\mathbf{cot}_{n-1}$ is not normal. 

For example, we can consider the truncated case of the category $Arr(\X)$ of arrows in $\X$ and $\mathbf{cot}_0=cok: Arr(\X)\to \X$ and the dihedral group $D_4=\langle a,b \mid a^2=b^4=0, aba=b^{-1} \rangle$. Let $X=<a>\to D_4$ be the inclusion. Then $\eta_X$ is given by the vertical morphisms in the diagram
\[\begin{tikzcd}<a>\ar[r, "X"]\ar[d, "\eta_{X,1}"] & D_4\ar[d, "\eta_{X,0}"] \\ 0\ar[r] & D_4/<a,b^2>\end{tikzcd}\]
where $ker(\eta_X)$ is the inclusion $<a>\to <a,b^2>$ which does not have trivial kernel, so $\mathbf{cot}_0(ker(\eta_X))$ is not trivial. 
\end{strc}

However, we obtain a torsion theory when restricted to the case of proper chains and also $\mathbf{cot}_{n-1}$ will be a normal functor.

\begin{lemma}  For each $n\in \mathbb{Z}$ the adjunction $\mathbf{cot}_{n-1}\dashv \mathbf{sk}_{n-1}$ can be restricted to proper chains:
\[\mathbf{cot}_{n-1}\dashv \mathbf{sk}_{n-1}: \begin{tikzcd}\pch\ar[r, bend left]\ar[r, phantom, "\perp"] & \pch_{n-1 \geq}\ar[l, bend left] \end{tikzcd}. \] 
\end{lemma}
\begin{proof}
It suffices to prove that the $(n-1)$-cotruncation of a proper chain complex $M$ is again proper. Indeed, consider the diagram
\[\begin{tikzcd} 
\dots\ar[r]&M_{n}\ar[r, "\delta_{n}"] & M_{n-1} \ar[rr, "\delta_{n-1}"] \ar[dd, "q_n", two heads] \ar[rd, two heads, "e_{n-1}"] & & M_{n-2} \ar[dd, "1"] \ar[rr, "\delta_{n-2}"] && \dots
\\ && & \delta_{n-1}(M_{n-1})\ar[dd, "q'_n", near start , crossing over] 
\ar[ru, "m_{n-1}"', hook] & & &
\\ \dots\ar[r] & 0\ar[r] & Cok(\delta_{n})\ar[rr, "\delta'_{n-1}", near start] \ar[rd, "e'_{n-1}"', two heads] & & M_{n-2}\ar[rr, "\delta_{n-2}"] && \dots
\\ & & &  \delta'_{n-1}(Cok(\delta_{n})) \ar[ru, "m'_{n-1}"'] &  && \end{tikzcd} \]
where $\delta'_{n-1}$ is induced by the cokernel $q_n=cok(\delta_{n})$ and consider $(e_{n-1},m_{n-1})$ and $(e'_{n-1}, m'_{n-1})$ the image factorizations of $\delta_{n-1}$ and $\delta'_{n-1}$. Since $m_{n-1}=m'_{n-1}q'_n$ is a monomosphism then $q'_n$ is both a monomorphism and a normal epimorphism, hence $q'_n$ is an isomorphism. Then $m'_{n-1}$ is a normal monomorphism.
\end{proof}

\begin{defi} We define the full subcategories in $\pch$ for each $n \in \Z$:
\[\mathcal{EP}_n=\{M \mid \mbox{$\delta_n$ is a normal epi and}\ M_i=0 \ \mbox{for}\ n-1>i \}.  \]
 And, similarly,
 \[\mathcal{MN}_n=\{M \mid \mbox{$\delta_n$ is a normal mono and}\ M_i=0 \ \mbox{for}\ i>n \}. \]
For instance, a proper chain complex $M$ in $\mathcal{EP}_n$ looks like this:
\[ \begin{tikzcd} \dots \ar[r]  & M_{n+1} \ar[r, "\delta_{n+1}"] & M_n\ar[r,two heads , "\delta_n"] & M_{n-1}\ar[r] & 0\ar[r] & 0\ar[r]& \dots \end{tikzcd} \]
with $\delta_n$ a normal epimorphism. Similarly, a proper chain complex $M$ in $\mathcal{NM}_n$ looks like this:
\[ \begin{tikzcd} \dots \ar[r] & 0\ar[r] & 0 \ar[r] & M_n\ar[r, hook,"\delta_n"] & M_{n-1}\ar[r, "\delta_{n-1}"] & M_{n-2}\ar[r] & \dots  \end{tikzcd} \]
where $\delta_n$ is a normal monomorphism.

In addition, we will also consider the full subcategory of $\chn$:
\[Ker(\mathbf{cot}_{n-1})=\{ M \mid \mathbf{cot}_{n-1}(M)\cong 0 \}.\]
\end{defi}

\begin{lemma}\label{lpch} Let $f:A \to B$ a morphism in $\X$. If $f$ is proper and has trivial cokernel, then $f$ is a normal epimorphism.
\end{lemma}
\begin{proof} Consider $e, m$ the normal epi/mono factorization of $f$ and the diagram
\[ \begin{tikzcd} A\ar[r, "f"]\ar[d, "e"'] & B \ar[r, "q"] & Cok(f) \\ f(A)\ar[ru, "m"] \ar[r,"m'"'] & Ker(q) \ar[u, "k"'] &  \end{tikzcd}\]
where $m'$ is induced by the kernel $Ker(q)$. Then if $f$ is a proper morphism then $m'$ is an isomorphism since the normal monomorphism $m$ is the kernel of its cokernel $q=cok(f)$. Also, if $Cok(f)=0$ then $k$ is also an isomorphism. Finally, $m$ is an isomorphism and $f$ is a normal epimorphism.
\end{proof}

\begin{lemma} The restriction of the subcategory $Ker(\mathbf{cot}_{n-1})$ to proper chains is equivalent to $\mathcal{EP}_n$: 
\[Ker(\mathbf{cot}_{n-1})\cap \pch = \mathcal{EP}_n\, .\]
\end{lemma}
\begin{proof} A chain complex $M$ in $Ker(\mathbf{cot}_{n-1})$ has all $M_i=0$ for $n-1>i$ and with the differential $\delta_n$ with trivial cokernel, so it follows from lemma \ref{lpch}.
\end{proof}

\begin{teo}\label{ttcok} For each $n\in \mathbb{Z}$ we have:
\begin{enumerate}
\item The pair $(Ker(\mathbf{cot}_{n-1}), \chn_{n-1 \geq})$ of subcategories in $\chn$ satisfy axiom TT1 of a torsion theory.
\item The pair $(\mathcal{EP}_n, \pch_{n-1 \geq})$ of subcategories in $\pch$ is a torsion theory in $\pch$.
\end{enumerate}
\end{teo}
\begin{proof} 1) Let $f:M\to N$ a morphism in $\chn$ with $M$ in $Ker(\mathbf{cot}_{n-1})$ and $N$ in $\chn_{n-1 \geq}$:
\[
\begin{tikzcd} 
\dots\ar[r] & M_{n+1}\ar[r, "\delta_{n+1}"]\ar[d, "f_{n+1}"] & M_n\ar[r, "\delta_n"]\ar[d, "f_n"] & M_{n-1}\ar[r]\ar[d, "f_{n-1}"] &0\ar[d, "f_{n-2}"]\ar[r] & \dots \\
\dots\ar[r] & 0\ar[r] & 0\ar[r] & N_{n-1}\ar[r, "\delta'_{n-1}"] & N_{n-2}\ar[r] & \dots
\end{tikzcd}\]
Since $\delta_n: M_n \to M_{n-1}$ has trivial cokernel then $f_{n-1}=0$, hence $f=0$.

From 1), we only need to prove axiom TT2 of a torsion theory. For a proper chain complex $M$ the short exact sequence is given by: 
\begin{equation}\label{sescok}
\begin{tikzcd}
\hdots \ar[r]
 & M_{n+1} \ar[r, "\delta_{n+1}"] \ar[d, "id"]
 & M_n \ar[r, "e_n", two heads] \ar[d, "id"]
  & \delta_n(X_n) \ar[r] \ar[d, "m_n", hook]
   & 0 \ar[r] \ar[d]
    & \hdots
 \\ 
\hdots \ar[r]
 & M_{n+1} \ar[r, "\delta_{n+1}"] \ar[d]
 & M_n \ar[r, "\delta_n"] \ar[d]
  & M_{n-1} \ar[r, "\delta_{n-1}"] \ar[d, "cok(\delta_n)", two heads] 
   & M_{n-2} \ar[r] \ar[d, "id"]
    & \hdots
\\
\hdots \ar[r]
 & 0 \ar[r]
 & 0 \ar[r] 
  & Cok(\delta_{n}) \ar[r]
   & M_{n-2} \ar[r]
    & \hdots
\end{tikzcd}
\end{equation}
Notice that since $M$ is a proper chain complex $Cok(\delta_n) \cong M_{n-1}/\delta_n(M_n)$.
\end{proof}

\begin{strc} By duality, let $\chn_{\geq n}$ be the category of $n$-truncated below chain complexes. It is straightforward to define the duals of the functors of \ref{adjun}:
\begin{itemize}
\item $\mathbf{tr}'_n( M )=  \begin{tikzcd}
\hdots \ar[r] & M_{n+2}\ar[r] & M_{n+1}\ar[r] & M_n
\end{tikzcd}$
\item $\mathbf{sk}'_n(M)= \begin{tikzcd} 
\hdots \ar[r]  & M_{n+1}\ar[r] & M_n\ar[r] & 0\ar[r]& 0\ar[r] &\hdots
\end{tikzcd}$
\item $\mathbf{cosk}'_n(M) = \begin{tikzcd}
\hdots \ar[r]  & M_{n+1}\ar[r] & M_n\ar[r, "cok(\delta_{n+1})"] & cok(\delta_{n+1})\ar[r]& 0 \hdots
\end{tikzcd}$
\item $\mathbf{cot}'_n(M)= \begin{tikzcd}
\hdots \ar[r] & M_{n+2}\ar[r] & M_{n+1}\ar[r, "\delta_{n+1}"] & ker(\delta_n)
\end{tikzcd}$
\end{itemize}
and indeed these also give a string of adjunctions:
\[ \mathbf{cosk}'_n \dashv \mathbf{tr}'_n \dashv \mathbf{sk}'_n \dashv \mathbf{cot}'_n: 
\begin{tikzcd}[row sep=large]
\chn_{ \geq n} 
\ar[d, bend left, shift right=.5] \ar[d, bend right, shift right=4] \ar[d, phantom, "\dashv"] \ar[d,phantom,"\dashv", shift right=5.5 ] \ar[d,phantom, "\dashv", shift left=5.5 ] \\
\chn 
\ar[u, bend left, shift right=.5] \ar[u, bend right, shift right=4]
\end{tikzcd} .
\]
We consider $\chn_{n \geq}$ as a subcategory of $\chn$ through $\mathbf{sk}'_n$.
However, unlike in the case of $\chn_{n \geq}$, the category $\chn_{\geq n}$ is a torsion category without the need of the restriction to proper chain complexes with $\mathbf{sk}'_n \dashv \mathbf{cot}'_n$ as the coreflection.
\end{strc}

\begin{teo}\label{ttker}  For each $n\in \mathbb{Z}$ we have:
\begin{enumerate}
\item The adjunction $\mathbf{tr}_{n-1} \dashv \mathbf{cosk}_{n-1}: \chn \to \chn_{n-1 \geq}$ is a localization and thus by theorem \ref{loctt} it induces a hederitary torsion theory $(\T_{\mathbf{tr}_{n-1}}, \F_{\mathbf{tr}_{n-1}})$ in $\chn$.
\item The category $\T_{\mathbf{tr}_{n-1}}$ is equivalent to $\chn_{ \geq n}$.
\item The reflector of $\chn_{ \geq n}$ is given by  $\mathbf{cot}'_n$.
\item The restriction of $(\chn_{ \geq n}, \F_{\mathbf{tr}_{n-1}})$ to $\pch$ is the torsion theory $(\pch_{\geq n}, \mathcal{MN}_n)$ in $\pch$.
\end{enumerate}
\end{teo}
\begin{proof}  1) The functor $\mathbf{tr}_{n-1}$ preserves finite limits since it admits a left adjoint, namely $\mathbf{sk}_{n-1}$. 2) is trivial since by definition $\T_{\mathbf{tr}_{n-1}}=Ker(\mathbf{cot}'_n)$. 3) and 4) follow immediately from the associated short exact sequence of $(\T_{\mathbf{tr}_{n-1}}, \F_{\mathbf{tr}_{n-1}})$, which is given by Theorem \ref{loctt}. 

To be precise, the unit $\eta$ of $\mathbf{tr}_{n-1} \dashv \mathbf{cosk}_{n-1}$ for a chain complex $M$ is:
\[\begin{tikzcd} 
M=\ar[d, "\eta_M"]  \dots\ar[r] & M_{n+1}\ar[r, "\delta_{n+1}"]\ar[d] & M_n\ar[r, "\delta_n"]\ar[d,"\delta_n'"] & M_{n-1}\ar[r, "\delta_{n-1}"]\ar[d,"1"] & M_{n-2}\ar[d, "1"]\ar[r] & \dots
\\ \mathbf{Cosk}_{n-1}(M)=  \dots\ar[r] & 0\ar[r]& Ker(\delta_{n-1})\ar[r, "ker(\delta_{n-1})"] & M_{n-1}\ar[r, "\delta_{n-1}"] & M_{n-2}\ar[r] & \dots  \end{tikzcd}\]
The normal epi-mono factorization of $\delta'_n$ is given by $(e_n, m'_n)$ where $m'_n$ is given by $m_n$ and the universal property of $ker(\delta_{n-1})$:
\[\begin{tikzcd} M_n \ar[r, "\delta_n"]\ar[d, "e_n", two heads] & M_{n-1}\ar[r, "\delta_{n-1}"] & M_{n-2} \\ \delta_n(M_n)\ar[ru,"m_n"]\ar[r, "m'_n"] & Ker(\delta_{n-1})\ar[u,"ker(\delta_{n-1})"', hook] & \end{tikzcd} \]
So the associated short exact sequence for a chain complex $M$ is: 
\begin{equation}\label{sesker}
\begin{tikzcd}
\hdots \ar[r]
 & M_{n+1} \ar[r] \ar[d, "id"]
  & ker(\delta_n) \ar[r] \ar[d, "ker(\delta_n)", hook]
   & 0 \ar[r] \ar[d]
   & 0 \ar[r] \ar[d]
    & \hdots
 \\ 
\hdots \ar[r]
 & M_{n+1}\ar[r, "\delta_{n+1}"] \ar[d]
  & M_n \ar[r, "\delta_n"] \ar[d, "e_n", two heads] 
   & M_{n-1} \ar[r, "\delta_{n-1}"] \ar[d, "id"]
    & M_{n-2} \ar[r] \ar[d, "id"]
    & \hdots
\\
\hdots \ar[r]
 & 0 \ar[r] 
  & \delta(M_n)\ar[r, "m_n"]
   & M_{n-1} \ar[r]
    & M_{n-2} \ar[r] 
    & \hdots
\end{tikzcd}
\end{equation}
\end{proof}

Now, we will now restrict ourselves to the case of proper chains.

\begin{defi}\label{cot}  For each $n \in \mathbf{Z}$, the torsion theories in $\pch$ from theorems \ref{ttcok} and \ref{ttker} will be denoted as
\[\mathcal{COK}_n=(\mathcal{EP}_n,\pch_{n-1\geq }) \quad \mbox{and} \quad \mathcal{KER}_n=(\pch_{ \geq n},\mathcal{MN}_n )\]
with the preradicals $ker_n, cok_n: \pch \to \pch$, respectively.

We will write $\mathcal{COT}(\X)$ for the set of torsion theories $\mathcal{COK}_n, \, \mathcal{KER}_n$ given by the cotruncation functors;
\[\mathcal{COK}_n = \begin{tikzcd} \mathcal{EP}_n\ar[r, bend left]\ar[r, phantom, "\perp"]  & \pch  \ar[l, bend left ]\ar[r, bend left, "\mathbf{cot}_{n-1}" above]\ar[r,phantom, "\perp"] &  \pch_{n-1\geq } \ar[l, bend left, "\mathbf{sk}_{n-1}" below] \end{tikzcd}\] and 
\[\mathcal{KER}_n=\begin{tikzcd}\pch_{ \geq n}\ar[r, bend left, "\mathbf{sk}_n'" above]\ar[r, phantom, "\perp"]  & \pch \ar[l, bend left, "\mathbf{cot}_n'" below]\ar[r, bend left]\ar[r,phantom, "\perp"] &  \mathcal{MN}_n \, . \ar[l, bend left] \end{tikzcd} \]
 The next result asserts that $\mathcal{COT}(\X)$ is a linearly order sublattice of $(\pch) tors$.
\end{defi}

\begin{prop}\label{latinf}  In $\pch$ for each $n \in \Z$ we have the embeddings of full subcategories 
 \[  \hdots \; \leq \; \pch_{ \geq n+1} \; \leq \; \mathcal{EP}_{n+1} \; \leq \; \pch_{ \geq n} \; \leq \; \mathcal{EP}_n  \; \leq \; \pch_{ \geq n-1} \; \leq \; \hdots \]
Equivalently,
\[  \hdots \; \geq \; \mathcal{MN}_{n+1} \; \geq \; \pch_{n+1\geq } \; \geq \; \mathcal{MN}_n \geq \; \pch_{n \geq } \; \geq \; \mathcal{MN}_{n-1} \; \geq \; \hdots \; .\]
Moreover, there is a linearly ordered lattice of torsion theories in $\pch$:
\[ O  \leq  \hdots \leq \mathcal{KER}_{n+1} \leq \mathcal{COK}_{n+1} \leq \mathcal{KER}_n \leq \mathcal{COK}_n \leq \hdots \leq \pch \]
\end{prop}
\begin{proof}
By definition we have $\mathcal{EP}_n \leq \pch_{ \geq n-1}$ and since a morphism $M_{n+1} \to 0$ is a normal epimorphism we have $\pch_{\geq n+1} \leq \mathcal{EP}_n$. Recall that the order is reverse for the torsion-free subcategories.
\end{proof}

This construction works with truncated or bounded chains complexes, in particularly we will be interested in the case for $\pch_{ \geq 0}$, and $\pch_{n \geq 0}$, the category of chain complexes bounded above $n$ and below $0$ for a fixed $n$.

\begin{coro}\label{lati0}  In $\pch_{ \geq 0}$ there is a linearly ordered lattice of torsion theories given by:
\[ O  \leq  \hdots \leq \mathcal{KER}_n \leq \mathcal{COK}_n \leq \hdots  \leq \mathcal{KER}_1 \leq \mathcal{COK}_1 \leq \pch_{ \geq 0} \]
\end{coro}

\begin{coro}\label{latn0}  In $\pch_{n \geq 0}$ there is a linearly ordered lattice of torsion theories given by:
\[ O \leq \mathcal{KER}_n \leq \mathcal{COK}_n \leq \hdots \leq \mathcal{COK}_2 \leq \mathcal{KER}_1 \leq \mathcal{COK}_1 \leq \pch_{n \geq 0} \]
\end{coro}

We will write $\mathcal{COT}(\pch_{ \geq 0})$ and $\mathcal{COT}(\pch_{n \geq 0})$ for the corresponding lattice of torsion theories in the bounded cases of $\pch_{\geq 0}$ and of $\pch_{n \geq 0}$.

\begin{strc} \textit{Example}. For the bounded case of $\pch_{2 \geq 0}$ the lattice $\mathcal{COT}(\pch_{2 \geq 0})$ is:
\[ O \leq \mathcal{KER}_2 \leq \mathcal{COK}_2 \leq \mathcal{KER}_1 \leq \mathcal{COK}_1 \leq \pch_{2 \geq 0}\, . \]
This lattice induces a lattice of preradicals $\pch_{2 \geq 0}$  and hence, for a fixed proper chain complex $M$, a lattice of torsion subobjects of $M$:
\[\begin{tikzcd}[row sep=tiny] 
   M=         & M_2\ar[r, "\delta_2"]  & M_1    \ar[r, "\delta_1"] & M_0 
\\ cok_1(M)=  & M_2\ar[r, "\delta_2"]  & M_1    \ar[r, "e_1", two heads] & \delta_1(M_1)
\\ ker_1(M)=  & M_2\ar[r, "\delta'_2"]  & ker(\delta_1) \ar[r]    & 0
\\ cok_2(M)=  & M_2\ar[r, "e_2", two heads]  & \delta_2(M_2) \ar[r]      & 0
\\ ker_2(M)=  & ker(\delta_2) \ar[r] & 0               \ar[r]  & 0
\\ 0=         & 0 \ar[r]             & 0              \ar[r]   & 0\, .
\end{tikzcd}  \]   
\end{strc}

\section{Homology}

In abelian categories, the $n$th-homology objects of a chain complex $M$ is usually defined as $H_n(M)= ker(\delta_n)/\delta_{n+1}(M_{n+1})$.  We can also consider the dual homology object $K_n$. In other words, consider the commutative diagram:
\[\begin{tikzcd} M_{n+1}\ar[d, "e_{n+1}"'] \ar[rr, "\delta_{n+1}"]& & M_n \ar[rr, "\delta_n"] \ar[rd, "q_{n+1}"'] \ar[rrd, "e_{n}"] && M_{n-1} 
 \\ \delta_{n+1}(M_{n+1})\ar[rru, "m_{n+1}"] \ar[r, "m'_{n+1}"'] & Ker(\delta_n)\ar[ru, "k_n"'] & & Cok(\delta_{n+1}) \ar[r, "e'_n"'] & \delta_n(M_n) \ar[u, "m_n"'] \end{tikzcd}\]
where $\delta_n=m_ne_n, \, \delta_{n+1}=m_{n+1}e_{n+1}$ are the normal epi/mono factorizations and $m'_{n+1}$ and $e'_n$ are induced by  $Ker(\delta_n)$ and $Cok(\delta_{n+1})$, respectively. Then we have 
\[H_n(M)=Cok(M_{n+1} \to Ker(\delta_n))=Cok(m'_{n+1})\]
and \[K_n(M)= Ker(Cok(\delta_{n+1}\to M_{n-1}))=Ker(e'_n).\] 
It is well-known that in abelian categories the objects $H_n(M)$ and $K_n(M)$ are naturally isomorphic, and in \cite{EveLin04} this was proved to be the case also for exact homological categories provided that the chain complex $M$ is proper. The following result provides an alternative proof of this fact as well as showing the connection of the objects $H_n(M), K_n(M)$ with the preradicals in $\mathcal{COT}$.

\begin{lemma}\label{homo} Let $M$ be a proper chain complex then the objects $H_n(M)$, $K_n(M)$ are isomorphic and are given by
\[H_n(M)\cong K_n(M) \cong ker_n(M)/cok_{n+1}(M) \]
where $cok_{n+1}(M)$, $ker_n(M)$ are the torsion subojects of $M$ given by the torsion theories as in Definition \ref{cot} and where $H_n(M)$, $K_n(M)$ are considered as trivial chain complexes except at the order $n$ that have the objects $H_n(M)$, $K_n(M)$ respectively.
\end{lemma}
\begin{proof}
Since $\chn$ is semi-abelian it follows from Theorem (\ref{n3t}) when we consider the normal subobjects  $cok_{n+1}(M) \leq ker_n(M)$ of $M$. So we have a short exact sequence in $\chn$: 
\[\begin{tikzcd}[column sep=scriptsize]  0\ar[r] & ker_n(M)/ cok_{n+1}(M)\ar[r] & M/ cok_{n+1}(M)\ar[r] &  M/ker_n(M)\ar[r] & 0 \end{tikzcd} .\]
To be more precise observe the short exact sequences that define $H_n(M)$ and $K_n(M)$.  Then $H_n(M)$ is the cokernel of the inclusion $cok_{n+1}(M) \leq ker_n(M)$:
\[
\begin{tikzcd} 
0\ar[d] &&&&&\\
cok_{n+1}(M)= \ar[d] & \hdots \ar[r]
 & M_{n+1}  \ar[r] \ar[d]
  & \delta_{n+1}(M_{n+1}) \ar[r] \ar[d, "m_{n+1}'", hook]
   & 0  \ar[r] \ar[d]     
     & \hdots     \\
ker_n(M)= \ar[d] & \hdots \ar[r]
 & M_{n+1} \ar[r] \ar[d]
  & ker(\delta_n) \ar[r] \ar[d]
   & 0  \ar[r] \ar[d]
     & \hdots  \\
H_n(M)\ar[d]=& \hdots \ar[r]
 & 0  \ar[r]
  & H_n(X) \ar[r]
   & 0 \ar[r]
     & \hdots \\ 0&&&&& 
\end{tikzcd}
\]
so $H_n(M)\cong ker_n(M)/cok_{n+1}(M)$; and on the other hand
\[
\begin{tikzcd} 
0 \ar[d]&&&&& \\
K_n(M) = \ar[d] & \hdots \ar[r]
 & 0  \ar[r] \ar[d]
  & K_n(M) \ar[r] \ar[d]
   & 0  \ar[r] \ar[d]     
     & \hdots     \\
M/cok_{n+1}(M) = \ar[d] & \hdots \ar[r]
 & 0 \ar[r] \ar[d]
  & cok(\delta_{n+1}) \ar[r] \ar[d, "e'_n", two heads]
   & M_{n-1}  \ar[r] \ar[d]
     & \hdots  \\
M/ker_n(M) = \ar[d]& \hdots \ar[r]
 & 0  \ar[r]
  & \delta_n(M) \ar[r]
   & M_{n-1} \ar[r]
     & \hdots \\ 0&&&&& 
\end{tikzcd}
\]
so $K_n(M)\cong ker(M/cok_{n+1}(M) \to M/ker_n(M))$. Theorem (\ref{n3t}) yields the isomorphism $H_n(M)\cong K_n(M)$.
\end{proof}

\begin{prop}  For a proper chain complex $M$ the following are equivalent:
\begin{enumerate}
\item $H_n(M)=0$ ;
\item $M/ker_{n+1}(M) \cong \mathbf{Cosk}_n(M)$.
\end{enumerate}
Similarly, the following are equivalent:
\begin{enumerate}
\item $H_n(M)=0$ ;
\item $cok_n(M) \cong \mathbf{Cosk}'_n(M)$.
\end{enumerate}
where $ker_{n+1}(M)$ and $cok_n(M)$ are the torsion subobjects of $M$ given by the torsion theories in \ref{cot}. 
\end{prop}
\begin{proof} First, recall from  Theorem \ref{ttker}  that the unit $M \to \mathbf{Cosk}_n(M)$ factors through the reflection of $\mathcal{MN}_{n+1}$:
\[
\begin{tikzcd}[column sep=scriptsize] 
M = \ar[d, two heads] & \hdots \ar[r]
 & M_{n+2}  \ar[r, "\delta_{n+2}"] \ar[d]
  & M_{n+1} \ar[r, "\delta_{n+1}"] \ar[d, "e_{n+1}" ,two heads]
   & M_n  \ar[r, "\delta_n"] \ar[d]     
     & \hdots     \\
M/ker_{n+1}(M) = \ar[d] & \hdots \ar[r]
 & 0 \ar[r] \ar[d]
  & \delta_{n+1}(M_{n+1}) \ar[r, "m_{n+1}",hook] \ar[d]
   & M_n  \ar[r, "\delta_n"] \ar[d]
     & \hdots  \\
\mathbf{Cosk}_n(M) = & \hdots \ar[r]
 & 0  \ar[r]
  & ker(\delta_n) \ar[r, "k(\delta_n)", hook]
   & M_n \ar[r, "\delta_n"]
     & \hdots 
\end{tikzcd}
\]
And by definition, $\delta_{n+1}(M_{n+1}) \cong ker(\delta_n)$ if and only if $H_n(M)=0$. The second part is similar, since $cok(\delta_{n+1}) \cong \delta_n(M)$ if and only if $H_n(M)=0$.
\end{proof}

From \cite{ill71}, it is known that the contruncation functor $\mathbf{cot}_n, \, \mathbf{cot}'_n$ give truncations in the homology objects. The following lemma generalises these facts.

\begin{lemma}\label{homology}  Let $ker_n, \, cok_n$ be the preradicals of the torsion theories in Definition \ref{cot}. For a proper chain complex $M$ we have:
\begin{enumerate}
\item For all $n>0$
\[H_i(cok_n(M))=H_i(ker_n(M))= \left\{\begin{array}{cc}
  H_i(M) & i \geq n \\
 0     & n >i \, .
 \end{array}\right. \]
\item For all $n>0$
\[H_i \left( \frac{M}{cok_n(M)} \right)= H_i\left( \frac{M}{ker_n(M)} \right)= \left\{\begin{array}{cc}
  0 & i \geq n \\
 H_i(M)     & n >i \, .
 \end{array}\right.\]
\item For all $n>0$ 
\[ H_i\left( \frac{cok_n(M)}{ker_n(M)} \right) = 0 \ \mbox{for all} \ i \, .\]
\item For $m>n$
\[ H_i\left( \frac{cok_n(M)}{ker_m(M)} \right) =H_i\left( \frac{cok_n(M)}{cok_m(M)} \right)= \left\{\begin{array}{cc}
  H_i(M) &  m > i \geq n \\
 0  & \mbox{otherwise} \, .
 \end{array}\right.\]
\item Moreover, for $m>n$ 
\[H_i\left( \frac{cok_n(M)}{cok_m(M)} \right)=H_i\left( \frac{ker_n(M)}{cok_m(M)} \right)=H_i\left( \frac{ker_n(M)}{ker_m(M)} \right)\, .\]
\item In particular, for $m=n+1$
\[ H_i\left( \frac{cok_n(M)}{ker_{n+1}(M)} \right)= \left\{\begin{array}{cc}
  H_i(M) & i=n \\
 0  & i\neq n \, .
 \end{array}\right.\]
\end{enumerate} 
\end{lemma}  
\begin{proof}
It is straightforward to calculate the homology of each chain complex.
For 1) consider:
\[\begin{tikzcd}[row sep=small] 
cok_n(M)= & \dots\ar[r] & M_{n+1}\ar[r, "\delta_{n+1}"] & M_n \ar[r, "e_n"]& \delta_n(M_n)\ar[r]   &0\ar[r] & \dots  
\\
ker_n(M)= & \dots\ar[r] & M_{n+1}\ar[r, "\delta_{n+1}'"] & ker(\delta_n) \ar[r]& 0\ar[r]   &0\ar[r] & \dots 
\end{tikzcd} \]

For 2) consider:
\[\begin{tikzcd}[row sep=small] 
\frac{M}{cok_n(M)}= & \dots\ar[r] & 0\ar[r] & 0 \ar[r]& \frac{M_n}{\delta_n(M_{n+1})}\ar[r, "\delta_{n-1}'"]   & M_{n-2}\ar[r] & \dots  
\\
\frac{M}{ker_n(M)}= & \dots\ar[r] & 0\ar[r] & \delta_n(M_n) \ar[r, "m_n"]& M_{n-1}\ar[r]   &M_{n-2}\ar[r] & \dots 
\end{tikzcd} \]

For 3) consider:
\[\begin{tikzcd}[row sep=small] 
\frac{cok_n(M)}{ker_n(M)}= & \dots\ar[r] & 0\ar[r] & \frac{M_n}{ker(\delta_n)} \ar[r, "\cong"]& \delta_n(M_n) \ar[r]  & 0\ar[r] & \dots
\end{tikzcd}
\]

For 4) and 5) the following chain complexes have the same homology:
\[
\begin{tikzcd}[row sep=small, column sep=scriptsize] 
\frac{ker_n(M)}{ker_m(M)}=   &\dots\delta_m(M_m) \ar[r] & M_{m-1} \ar[r]& \dots \ar[r] & M_{n+1} \ar[r] & ker(\delta_n)\ar[r] &0\dots \\
\frac{ker_n(M)}{cok_m(M)}=   & \dots 0 \ar[r] & \frac{M_{m-1}}{\delta_m(M_m)} \ar[r]& \dots \ar[r] & M_{n+1} \ar[r] & ker(\delta_n)\ar[r] &0\dots \\  
\frac{cok_n(M)}{ker_m(M)}=   &\dots\delta_m(M_m) \ar[r] & M_{m-1} \ar[r]& \dots \ar[r] & M_{n+1} \ar[r] & M_n \ar[r]& \delta_n(M_n)\dots  \\
\frac{cok_n(M)}{cok_m(M)}=   &\dots 0 \ar[r] & \frac{M_{m-1}}{\delta_m(M_m)} \ar[r]& \dots \ar[r] & M_{n+1} \ar[r] &  M_n \ar[r]& \delta_n(M_n)\dots  
\end{tikzcd}
\]
\end{proof}

\section{Torsion theories induced by $tr_n \dashv cosk_n$}

We define the torsion theories $\mu_{\geq n}$, simplicial analogues for $\mathcal{MN}_n$ which similarly to the case of chain complexes, are defined by a localization $tr_n \dashv cosk_n$ in simplicial objects. The torsion category of $\mu_{\geq n}$ is the category of simplcial groups such that they are trivial below degree $n$. First, we recall some basic properties of simplicial objects.

Following \cite{GabZis67}, the simplicial category $\Delta$ has, as objects, finite ordinals $[n]=\{0,1,\dots n\}$ and as morphisms  monotone functions. In particular, we have the morphisms $\delta^n_i:[n-1] \to [n]$ the  injection which does not take the value $i\in [n]$ and $\sigma^n_i:[n+1] \to [n]$ the surjection where $\sigma(i)=\sigma(i+1)$. Any morphism $\mu$ in $\Delta$ can written uniquely as: \[\mu=\delta_{i_s}^n\delta_{i_s-1}^{n-1}\dots\delta^{n-t+1}_{i_1}\sigma^{m-t}_{j_t}\dots\sigma_{j_2}^{m-2}\sigma_{j_1}^{m-1}\, .\]
such that $n \geq i_s>\dots >i_1\geq 0$, $0\leq j_t< \dots <j_1<m$ and $n=m-t+s$.

The category of simplicial objects in a category $\X$ is the functor category $Simp(\X)=[\Delta^{op}, \X]$. Thus, a simplicial object $X: \Delta^{op}\to \X$ corresponds to: a family of objects $\{X_n\}_{n\in \mathbb{N}}$ in $\X$, the \textit{face morphisms} $d_i: X_n\to X_{n-1}$ and the \textit{degeneracies morphisms} $s_i: X_n \to X_{n+1}$ satisfying the \textit{simplicial identities}:
\begin{align*}
 d_i d_j &= d_{j-1} d_i   \quad   \mbox{if} \quad i<j  \\ 
 s_i s_j &= s_{j+1} s_i   \quad  \mbox{if} \quad i\leq j \\
 d_i s_j &= 
 \left\{\begin{array}{cc}
  s_{j-1} d_i & \mbox{if} \quad i<j  \\
 1  & \mbox{if} \quad i=j\, \mbox{or}\,\ i=j+1 \\
  s_j d_{i-1} & \mbox{if} \quad i>j+1\, .
 \end{array}\right. 
\end{align*} 

\begin{strc} Let $X$ be a simplicial object in a pointed category with pullbacks. The Moore normalization functor $N: Simp(\X) \to \chn$ is given by $N(X)_0=X_0$,
\[N(X)_n = \bigcap_{i=0}^{n-1} ker(d_i: X_n \to X_{x-1})\]
and differentials $\delta_n= d_n \circ \cap_i ker(d_i): N(X)_n \to N(X)_{n-1}$ for $n \geq 1$.

The functor $N$ preserves finite limits. In \cite{EveLin04}, it was proved that if $\X$ is a semi-abelian category $N$ also preserves regular epimorphisms, and hence, it preserves short exact sequences. Moreover, for a simplicial object $X$ the Moore complex $N(X)$ is a proper chain complex and we can define the $n$-homology object of a simplicial object $X$ as:
\[H_n(X)= H_n(N(X)). \]
The objects $H_n(N(X))$ are internal abelian objects for $n \geq 1$ (see \cite{EveLin04}). 

 This generalises the results proven by Moore for the case of simplicial groups, where the homotopy groups of a simplicial group $X$ can be calculated as $\pi_n(X)= H_n(N(X))$ (see, for example \cite{May67}).
\end{strc}

\begin{strc} If $\Delta_n$ is the full subcategory of $\Delta$ with objects $[m]$ for $m \leq n$, an $n$-truncated simplicial object $X$ in $\X$ is a functor $X: \Delta_n^{op} \to \X$. Let $Simp_n(\X)$ be the category of $n$-truncated simplicial objects, then there is truncation functor:
\[tr_n: \begin{tikzcd} Simp(\X)\ar[r] & Simp_n(X) \end{tikzcd}\]
which simply forgets the objects $X_i$  and the morphisms $s_i,\, d_i$ for degrees $i> n$.

It is a standard application of Kan extensions that if $X$ is finitely complete/cocomplete (as is the case if $\X$ is semi-abelian) each functor $tr_n$ admits a left/ right adjoint named the $n$-skeleton and $n$-coskeleton, respectively: $sk_n \dashv tr_n \dashv cosk_n$.  We will write $Sk_n=sk_n tr_n$ and $Cosk_n=cosk_ntr_n$. 

If $\X$ has finite limits, the $n$-coskeleton of an $n$-truncated simplicial object $X$ is described as follows. For an $n$-truncated simplicial object $X$ the \textit{simplicial kernel} of the face morphisms $d_0, \dots, d_n: X_n \to X_{n-1}$ is an object $\Delta_{n+1}$ with morphisms $\pi_0, \dots, \pi_{n+1}: \Delta_{n+1} \to X_n$ such that $d_i\pi_j=d_{j-1}\pi_i$ for all $i<j$ and it is universal with this property: given a family of morphisms $p_0, \dots, p_{n+1}: Y \to X_n$ such that $d_ip_j=d_{j-1}p_i$ for all $i<j$ then there is a unique morphism $\alpha:Y \to \Delta_{n+1}$ such that $\pi_i\alpha= p_i$. Moreover, the universal property of the simplicial kernel $\Delta_{n+1}$ allows to define degeneracies morphisms $s_i: X_n \to \Delta_{n+1}$. So, the simplicial kernel of $X$ defines an $(n+1)$-truncated simplicial object. Finally, the $cosk(X)$ is given by the iteration of successive simplicial kernels. 

For $n=0$, the $0$-coskeleton is known as the indiscrete functor $Ind: \X \to Simp(\X)$ given by:
\[\begin{tikzcd}[row sep=huge, column sep=large] 
Ind(X)=  & \dots\ar[r, shift left=3, "\pi_4"]\ar[r, shift right=3, "\pi_0"']\ar[r, phantom, "\vdots", shift left] 
&  X^4\ar[r, shift left=3, "\pi_4"]\ar[r, shift right=3, "\pi_0"']\ar[r, phantom, "\vdots", shift left] 
& X^3\ar[r, shift left=3, "\pi_4"]\ar[r, shift right=3, "\pi_0"']\ar[r, phantom, "\vdots", shift left] & X^2  \ar[r, shift left=3, "\pi_1"]\ar[r, shift right=3, "\pi_0"'] 
& X\, . \ar[l, "s_0"] \end{tikzcd} \] 
where $X^n$ is $n$-fold product of $X$ and the degeneracies are defined by the product projections.

For $\X=Grps$, the simplicial kernel $\Delta_{n+1}$ of a $n$-truncated simplicial group $X$ can be described as the subgroup of $X_n^{n+2}$ of $(n+2)$-tuples $(x_0, \dots, x_{n+1})$ such that $\pi_i(x_j)=\pi_{j-1}(x_i)$ for $i<j$ and where $\pi_i$ are the product projections.
\end{strc}

The following result was first proved for simplicial groups in \cite{Con84}. The generalization is straighforward. 

\begin{teo}\label{norcosk} Let $\X$ be a pointed category with finite limits. For a simplicial object $X$ with corresponding Moore complex $M$, then the Moore complex of $n$-coskeleton $Cosk_n(X)$ satifies:
\begin{itemize}
\item $N(Cosk_n(X))_i= M_i$ for $n \leq i$;
\item $N(Cosk_n(X))_{n+1}= ker(\delta_n: M_n \to M_{n+1})$;
\item $N(Cosk_n(X))_i=0$ for $i> n+1$.
\end{itemize}
\end{teo}

In other words, the Moore normalization $N$ commutes up to isomorphism with the coskeleton functors for simplicial objects and chain complexes:
\[\begin{tikzcd} Simp(\X)\ar[r, "N"]\ar[d, "tr_n"', bend right]\ar[d, phantom, "\dashv"] & \chn \ar[d, "\mathbf{tr}_n"', bend right]\ar[d, phantom, "\dashv"] \\ Simp_n(X)\ar[u, "cosk_n"', bend right]\ar[r,"N"] & \chn_{n \geq}\ar[u,"\mathbf{cosk}_n"', bend right] \end{tikzcd} \]

\begin{defi} Let $\X$ be a semi-abelian category. For each $n>1$ we have that $tr_{n-1} \dashv cosk_{n-1}$ is a localization since $tr_{n-1}$ admits a left adjoint, namely $sk_{n-1}$. Then, by Theorem \ref{loctt}, it defines a hereditary torsion theory in $Simp(\X)$ which will be written as:
\[\mu_{\geq n}:=(\T_{tr_{n-1}}, \F_{tr_{n-1}}).\]
We will also write $\mu_{\geq n}: Simp(\X)\to Simp(\X)$ for the associated idempotent radical.

By definition, the category $\T_{tr_{n-1}}$ is the full subcategory of simplicial objects $X$ with $X_i=0$ for $ n-1 \geq i$.
\end{defi}
 
The torsion theory $\mu_{\geq 1}=(\T_{tr_{0}}, \F_{tr_{0}})$ naturally extends the torsion theory $(Ab(\X), Eq(\X))$ in $Grpd(\X)$, to this end we need to recall a result about Mal'tsev categories.
 
\begin{strc} A category $\X$ is called a Mal'tsev category if any internal reflexive relation is an equivalence relation. Semi-abelian categories are Mal'tsev categories. 
\end{strc} 

\begin{prop}\label{duv} \cite{Duv19} Let $\X$ be a regular Mal'tsev category. The category $Grpd(\X)$ is closed under subobjects in $Simp(\X)$. 
\end{prop}

\begin{prop} Let $\X$ be a semi-abelian category and, for $n=0$, consider the torsion theory $\mu_{\geq 1}=(\T_{tr_0}, \F_{tr_0})$ in $Simp(\X)$. The torsion-free subcategory $\F_{tr_0}$ is equivalent to $Eq(\X)$. On the other hand, $Ab(\X) \subset \T_{tr_0}$.
\end{prop}
\begin{proof} 
For $n=0$ we have that $tr_0 \dashv cosk_0  \cong ()_0 \dashv Ind$. The unit of $()_0 \dashv Ind$ for a simplicial object $X$ is given by
\[\begin{tikzcd}[row sep=huge, column sep=large] 
X=\ar[d, "\eta_X"]  & \dots\ar[r, shift left=3, "d_4"]\ar[r, shift right=3, "d_0"']\ar[r, phantom, "\vdots", shift left] 
& X_3\ar[r, shift left=3, "d_3"]\ar[r, shift right=3, "d_0"']\ar[r, phantom, "\vdots", shift left] \ar[d, "{(d_0,d_1,d_2,d_3)}"] 
& X_2\ar[r, shift left=3, "d_2"]\ar[r, shift right=3, "d_0"'] \ar[r, phantom, "\vdots", shift left]\ar[d, "{(d_0,d_1,d_2)}"] 
& X_1\ar[r, shift left=3, "d_1"]\ar[r, shift right=3, "d_0"']\ar[d, "{(d_0,d_1)}"]
& X_0\ar[l, "s_0"]\ar[d, "1"] 
\\ Ind(X)=  & \dots\ar[r, shift left=3, "\pi_4"]\ar[r, shift right=3, "\pi_0"']\ar[r, phantom, "\vdots", shift left] 
&  X^4_0\ar[r, shift left=3, "\pi_4"]\ar[r, shift right=3, "\pi_0"']\ar[r, phantom, "\vdots", shift left] 
& X_0^3\ar[r, shift left=3, "\pi_4"]\ar[r, shift right=3, "\pi_0"']\ar[r, phantom, "\vdots", shift left] & X_0^2  \ar[r, shift left=3, "\pi_1"]\ar[r, shift right=3, "\pi_0"'] 
& X_0\, . \ar[l, "s_0"] \end{tikzcd} \] 
Since $\X$ is a Mal'tsev category and by Proposition \ref{duv}, if $X$ has a monic unit $\eta_X$ then $X$ is a groupoid since $Ind(X)$ is an equivalence relation. Finally, since $(d_0,d_1): X_1 \to X_0^2$ is monic then $X$ is an equivalence relation and $\F_{tr_0} \subseteq Eq(\X)$. Conversely, an equivalence relation always has a monic unit $\eta_X$.

On the other hand, an internal abelian group $X$ has $X_0=0$, so $Ab(\X) \subset \T_{tr_0}$.
\end{proof}

\begin{teo} Let $\X$ be a semi-abelian category and $X$ a simplicial object with Moore chain complex $M$. Then, the normalization functor $N$ maps the short exact sequence of $X$ given the torsion theory $\mu_{\geq n+1 }$ in $Simp(X)$ into the short exact sequence of $M$ given by the torsion theory $\mathcal{KER}_{n+1}$ in $\pch$. 

Moreover, $N$ maps the torsion theory $\mu_{\geq n+1}$ into $\mathcal{KER}_{n+1}$:
\[ \begin{tikzcd}[row sep=large]
\T_{tr_n}\ar[r, bend left]\ar[r, phantom, "\perp"]\ar[d, "N"] & Simp(\X)\ar[r, bend left] \ar[l, bend left]\ar[r, phantom, "\perp"]\ar[d, "N"]& \F_{tr_n}\ar[l, bend left]\ar[d, "N"]
\\ \pch_{\geq n+1}\ar[r, bend left] \ar[r, phantom, "\perp"]& \pch \ar[r, bend left]\ar[l, bend left]\ar[r, phantom, "\perp"]& \mathcal{MN}_n\ar[l, bend left]
\end{tikzcd} \]
i.e. the subcategory $\T_{tr_n}$ is mapped into $\pch_{\geq n+1}$ and $\F_{tr_n}$ into $\mathcal{MN}_n$. 
\end{teo}
\begin{proof}
Since $\X$ is semi-abelian the normalization functor $N$ preserves short exact sequences and also preserves the normal epi/mono factorization of morphisms in $Simp(X)$. Since $N$ commutes (up to isomorphism) with the truncation and coskeleton functors we have that for a simplicial object $X$ and its Moore complex $M$, the functor $N$ maps the short exact sequence in $Simp(\X)$:
\[ \begin{tikzcd} 0\ar[r] & ker(\eta_X) \ar[r, "k"] & X \ar[r, "e"] \ar[rd, "\eta_X"'] & \eta_X(X)\ar[d, "m"]\ar[r] &0 \\ && & Cosk_n(X)& \end{tikzcd} \]
into
\[
\begin{tikzcd}[column sep=scriptsize] 
N(ker(\eta_X)) = \ar[d, hook, "m(k)"] & \hdots \ar[r]
 & M_{n+2}  \ar[r, "\delta_{n+2}"] \ar[d]
  & ker(\delta_{n+1}) \ar[r] \ar[d, "e_{n+1}" ,two heads]
   & 0  \ar[r] \ar[d]     
     & \hdots     \\
N(X) = \ar[d, two heads, "M(e)"] \ar[dd, "N(\eta_x)"', bend right, shift right=10] & \hdots \ar[r]
 & M_{n+2}  \ar[r, "\delta_{n+2}"] \ar[d]
  & M_{n+1} \ar[r, "\delta_{n+1}"] \ar[d, "e_{n+1}" ,two heads]
   & M_n  \ar[r, "\delta_n"] \ar[d]     
     & \hdots     \\
N(\eta_X(X)) = \ar[d, "M(m)"] & \hdots \ar[r]
 & 0 \ar[r] \ar[d]
  & \delta_{n+1}(N_{n+1}) \ar[r, "m_{n+1}",hook] \ar[d]
   & M_n  \ar[r, "\delta_n"] \ar[d]
     & \hdots  \\
N(Cosk_n(X)) = & \hdots \ar[r]
 & 0  \ar[r]
  & ker(\delta_n) \ar[r, "k(\delta_n)", hook]
   & M_n \ar[r, "\delta_n"]
     & \hdots 
\end{tikzcd}\]

Since the short exact sequence of the torsion theories is preserved from $\mu_{\geq n+1 }$ to $\mathcal{KER}_{n+1}$, it follows that $N(\T_{tr_n}) \subseteq \pch_{\geq n+1}$ and $N(\F_{tr_n}) \subseteq \mathcal{MN}_n$. 
\end{proof}

\section{Torsion theories of truncated Moore complexes in simplicial groups}

In order to define the torsion theories $\mu_{n \geq}$, analogues for the torsion theories $\mathcal{COK}_n$ in simplicial objects, we will restrict ourselves to the case of the category of $\X=Grps$. With this stronger assumption the subcategories $\mathcal{M}_{ \geq n}$ and $\mathcal{M}_{n \geq}$ of simplicial groups with trivial Moore complex below/above at order $n$ will appear as torsion/torsion-free subcategories of the torsion theories $\mu_{n \geq}$ and $\mu_{\geq n}$, respectively. 

We need to recall some results of D. Conduch\'e \cite{Con84}. In particular, in a simplicial group $X$ each $X_n$ can be decomposed as successive semi-direct products of the objects of its Moore complex $M_i$ with $i \leq n$. 

\begin{strc}\label{ordersn} \cite{Con84} In order to avoid multiple subscripts we will write $\sigma_i= \bar{i}$ for the degeneracy maps of $\Delta$.
 
 For any  object $[n]=\{0<1<\dots <n\}$ of the simplicial category $\Delta$ we will introduce an order in $S(n)$ the set of surjective maps of $\Delta$ with domain $[n]$.  Any surjective map $\sigma: [n] \to [m]$ is written uniquely as $\sigma=\bar{i_1}\bar{i_2} \dots \bar{i}_{n-m}$ with $i_1<i_2< \dots <i_{n-m}$. We introduce the inverse lexicographic order in $S(n,m)$ the set of surjective maps form $[n]$ to $[m]$: 
\[\bar{i_1}\bar{i_2} \dots \bar{i}_{n-m} < \bar{j_1}\bar{j_2} \dots \bar{j}_{n-m} \quad  \mbox{if} \quad i_{n-1}=j_{n-m}, \dots, i_{s+1}=j_{s+1}, \\\ \mbox{and} \\\ i_s>j_s. \]  
This order extends to $S(n)$ by setting $S(n,m)< S(n,l)$ if $m>l$. 

As an example, for $S(4)$ we have:

\[id_{[4]}<  \bar{3} < \bar{2} < \bar{2}\bar{3} < \bar{1}< \bar{1}\bar{3} < \bar{1}\bar{2}< \bar{1}\bar{2}\bar{3} < \bar{0} <\bar{0}\bar{3} < \bar{0}\bar{2} < \bar{0}\bar{2}\bar{3} < \bar{0}\bar{1}< \bar{0}\bar{1}\bar{3} < \bar{0}\bar{1}\bar{2}< \bar{0}\bar{1}\bar{2}\bar{3}  \]

For a simplicial group $X$ with Moore complex $M$ and a surjective map $\mathbf{i}=\bar{i_1}\bar{i_2} \dots \bar{i}_{r}$ we have $s_\mathbf{i}=s_{i_r}  \dots s_{i_1}$ and $d_\mathbf{i}=d_{i_1} \dots d_{i_r}$. Using the order of $S(n)$ we have a filtration of $X_n$ by the subgroups

\[G_{n, \mathbf{i}}= \bigcap_{\mathbf{j} \geq \mathbf{i}} ker(d_\mathbf{j}).\]
Notice that $G_{n, id}=0$ and $G_{n, \bar{n-1}}=M(X)_n$.
 The order $S(n)$ satisfies for a surjective map $\mathbf{i}:[n] \to [r]$ and its  successor $\mathbf{j}$ we have the semidirect product
\[G_{n, \mathbf{j}} \cong G_{n, \mathbf{i}} \rtimes_{s_\mathbf{i}} M_r. \]
Finally, this implies that $X_n$ decomposes as a sequence of semi-direct products:
\[ X_n=( \dots( M_n \rtimes_{s_{n-1}} M_{n-1}) \rtimes_{s_{n-2}} \dots)\rtimes_{s_{p-1}\dots s_{0}} M_0  \]
\end{strc}

\begin{coro} For each $n \in \mathbb{N}$, the category $\mathcal{M}_{\geq n+1}$ and $\T_{tr_n}$ are equivalent. 

Moreover, the category $\mathcal{M}_{\geq n+1}$ is a torsion subcategory in $Simp(\X)$;
\[(\T_{tr_n}, \F_{tr_n}) \cong (\mathcal{M}_{\geq n+1}, \F_{tr_n}). \]
\end{coro}
\begin{proof} It follows immediately from the semidirect decomposition that a simplicial group $X$ has $X_i=0$ for $n>i$ if and only if $M_i=0$ for $n>i$.
\end{proof}

The analogue of the cotruncation functor for simplicial groups was introduced by T. Porter as follows.

\begin{strc} \cite{Por93}\label{porter} There is a cotruncation functor \[Cot_n: Simp(Grp) \to Simp(Grp)\] such that
\begin{equation}\label{cotnorm}
 \begin{tikzcd} Simp(Grp) \ar[d, "Cot_n"]\ar[r,"N"] & chn(Grp)\ar[d, "\mathbf{Cot}_n"]\\ Simp(Grp) \ar[r, "N"] & chn(Grp) \end{tikzcd} 
\end{equation}
commutes up to natural isomorphism, where $N$ is the Moore normalization functor. 
The functor $Cot_n(X)$ is defined as follows:
\[\begin{tikzcd}[row sep=tiny]  Cot_n(X)_i=X_i\quad \mbox{for} \quad n>i\, , \\ Cot_n(X)_n=X_n\, ,  \end{tikzcd} \]
and for $i>n$ the object $Cot_n(X)_i$ is obtained by deleting all $M_k$ for $k>n$ and replacing $M_n$ by $M_n/\delta_{n+1}(M_{n+1})$ in the semi-direct decomposition. 
\end{strc}

We recall some useful properties of this functor.

\begin{prop} \cite{Por93}\label{porter2} Let $\mathcal{M}_{n \geq}$ be the full subcategory of $Simp(Grp)$ defined  by those simplicial groups whose Moore complex is trivial for dimensions greater than $n$. Let $i_n: \mathcal{M}_{n \geq} \to Simp(Grp)$ the inclusion functor then
\begin{enumerate}
\item $Cot_n$ is left adjoint of $i_n$;
\item the unit $\eta_X: X \to Cot_n(X)$ of the adjunction  is a regular epimorphism which induces an isomorphism in $\pi_i(X)$ for $i \leq n$;
\item for any simplicial group $X$, $\pi(Cot_n(X))=0$ for $i>n$;
\item the inclusion $\mathcal{M}_{n \geq} \to \mathcal{M}_{n+1 \geq}$ correspond to a natural epimorphism \[\eta_n: Cot_{n+1} \to Cot_{n}\] and, for a simplicial group $X$, then $Ker(\eta_n(X))$ is a $K(\pi_{n+1}(X), n+1)$-simplicial group (an Eilenberg-Mac Lane simplicial group). 
\end{enumerate}
\end{prop}

This cotruncation functor for simplicial groups is normal and thus defines a torsion theory in $Simp(Grp)$ as in Theorem \ref{tfnormal}.

\begin{coro} The subcategory $\mathcal{M}_{n\geq }$ of $Simp(Grp)$ given by the simplicial groups with trivial Moore complex for dimension greater than $n$ is a torsion-free subcategory of $Simp(Grp)$. The torsion theory is given by the pair
\[\mu_{n \geq}=(Ker(Cot_n), \mathcal{M}_{n \geq}).\]
\end{coro}

\begin{proof} By Theorem \ref{tfnormal} it suffices to prove that the functor $Cot_n$ is normal. Let $\eta$ be the unit as in \ref{porter2}, for a simplicial group $X$ with a Moore complex $M$. Since taking normalization preserves short exact sequences we have that the Moore complex of $ker(\eta_X)$ is:
\[\begin{tikzcd} \dots\ar[r] & M_{n+2}\ar[r, "d_{n+2}"] & M_{n+1}\ar[r, "e_{n+1}", two heads]& d_{n+1}(M_{n+1})\ar[r] & 0\ar[r]  & \dots \end{tikzcd}\]
which is trivial under the chain cotruncation $\mathbf{cot}_n$. Since the functor $\mathbf{Cot}_n$ and $Cot_n$ commute with the Moore normalization as in \ref{norcosk} we have that $Cot_n(ker(\eta_X))=0$ for any simplicial group $X$.
\end{proof}

\begin{defi} For each $n$ we will denote $\mu_{n \geq}$ the torsion theory in $Simp(Grp)$ given by the functor $Cot_n$, i.e.:
\[\mu_{n \geq}=(Ker(Cot_n), \mathcal{M}_{n \geq})\]
 and also the associated idempotent radical will be denoted by $\mu_{n \geq}: Simp(\X) \to Simp(\X)$.
\end{defi}

The category $\mathcal{M}_{0 \geq}$ is equivalent to the category $Dis(Grp)$ of discrete simplicial groups, simplicial groups where all degenerecies and face morphisms are the identity.  And it is well-know from Loday's article \cite{Lod82} that the category $\mathcal{M}_{1 \geq}$ is equivalent to the category of internal grupoids $Grpd(Grp)$.

\begin{coro}\label{disgrpdtf} The categories $Dis(Grp)$ of discrete simplicial groups and the category $Grpd(Grp)$ of internal grupoids are torsion-free subcategories of $Simp(\X)$.
\end{coro}

\begin{teo} Let be $X$ a simplicial group with Moore complex $M$. The normalization functor $N$ maps the short exact sequence of $X$ given by the torsion theory $\mu_{n \geq}$ into the short exact sequence of $M$ given by $\mathcal{COK}_{n+1}$.

Moreover, $N$ maps the torsion category $Ker(Cot_n)$ into the torsion category $\mathcal{EP}_{n+1}$ and, respectively, the torsion-free category $\mathcal{M}_{n \geq}$ into $pch(Grp)_{n \geq}$;
\[ \begin{tikzcd}[row sep=large]
Ker(Cot_n)\ar[r, bend left]\ar[r, phantom, "\perp"]\ar[d, "N"] & Simp(\X)\ar[r, bend left] \ar[l, bend left]\ar[r, phantom, "\perp"]\ar[d, "N"]& \mathcal{M}_{n \geq}\ar[l, bend left]\ar[d, "N"]
\\ \mathcal{EP}_{n+1} \ar[r, bend left] \ar[r, phantom, "\perp"]& \pch \ar[r, bend left]\ar[l, bend left]\ar[r, phantom, "\perp"]& pch(Grp)_{n \geq}\ar[l, bend left]
\end{tikzcd} \]
\end{teo}
\begin{proof}
Since the cotruncation functors commute up to isomorphism with normalization as in diagram (\ref{cotnorm}) and $N$ preserves short exact sequences,  the short exact sequence in $Simp(Grp)$: 
\[\begin{tikzcd} 0\ar[r]& ker(\eta_X)\ar[r] &  X\ar[r, "\eta_X"] & Cot_n(X)\ar[r]&0\end{tikzcd} \] 
is mapped under $N$ into the short exact sequence (written vertically) in $pch(Grp)$:
\[\begin{tikzcd}[column sep=small]
N(ker(\eta_x))=\ar[d]&\dots\ar[r] & M_{n+1}\ar[d]\ar[r, "e_{n+1}" ] & \delta_{n+1}(M_{n+1})\ar[d, "m_{n+1}"]\ar[r]  & 0\ar[d] \ar[r] &  \dots \\ 
M=\ar[d]&\dots\ar[r] & M_{n+1}\ar[d]\ar[r, "\delta_{n+1}"] & M_n \ar[d]\ar[r, "\delta_n"]& M_{n-1}\ar[d]  \ar[r]& \dots\\
N(Cot_n(X))=&\dots\ar[r]&0\ar[r]  & M_n/\delta_{n+1}(M_{n+1})\ar[r, "\delta_n'"] & M_{n-1}  \ar[r]& \dots\, .
\end{tikzcd}\]

Since the associated short exact sequence of the torsion theory is preserved it follows that $N(Ker(Cot_n)) \subset \mathcal{EP}_{n+1} $ and $N(\mathcal{M}_{n \geq}) \subset pch(Grp)_{n \geq}$.
\end{proof}

\begin{teo} The torsion subcategories of the torsion theories $\mu_{n \geq}$ and $\mu_{\geq n+1}$ in $Simp(Grp)$ are linearly ordered as:
\[0 \subseteq \dots  \subseteq  Ker(Cot_{n+1}) \subseteq  \mathcal{M}_{\geq n+1} \subseteq  Ker(Cot_n) \subseteq  \mathcal{M}_{\geq n} \subseteq   \dots  \subseteq Simp(Grp).\]

 Moreover, the torsion theories $\mu_{n \geq}$ and $\mu_{\geq n+1}$ form a linearly ordered lattice $\mu(Grp)$: 
\begin{multline*}
0  \leq \dots \leq \; \mu_{ n+1\geq}  \;  \leq \;   \mu_{\geq n+1} \;  \leq \; \mu_{n\geq} \;   \leq \;  \mu_{\geq n} \;  \leq \dots \\  \dots \leq \; \mu_{\geq 2} \; \leq  \; \mu_{1\geq} \; \leq \;  \mu_{\geq 1} \;  \leq  \; \mu_{0\geq} \;   \leq Simp(Grp)\, . 
\end{multline*}
\end{teo}

\begin{proof}
First we will prove $\mathcal{M}_{\geq n+1} \subseteq Ker(Cot_n)$. For a simplicial group $X$ and $M$ its Moore normalization and $\eta$ the unit as in Proposition \ref{porter2}. Then, if $M_i=0$ for $n \geq i$ then $X_i=0$ for $n \geq i$ and since $\eta_X$ is a normal epimorphism then $Cot_n(X)_i=0$ for $n \geq i$. It follows from the semi-direct decomposition that $Cot_n(X)=0$.

Now we prove $Ker(Cot_n) \subseteq \mathcal{M}_{\geq n}$. From  it is clear that if $Cot_n(X)=0$ we have $X_i=0$ for $ n-1 \geq i$ then $M_i=0$ for $ n-1 \geq i$ and $X$ is in $\mathcal{M}_{\geq n}$.
\end{proof}

\begin{defi} We will write $\mu(Grp)$ for the linearly order lattice of torsion theories in $Simp(Grp)$ given by $\mu_{n \geq}$ and $\mu_{\geq n}$. 
\end{defi}

\begin{teo} The torsion theory $\mu_{n \geq }$ is hereditary and $\mu_{\geq n}$ is cohereditary. Moreover, the subcategories $\mathcal{M}_{n \geq}$ and $\mathcal{M}_{\geq n}$ are semi-abelian.
\end{teo}
\begin{proof}
It follows from the fact that $N$ is an exact functor. Then, $\mathcal{M}_{n \geq}$ and $\mathcal{M}_{\geq n}$ are semi-abelian by  theorem \ref{hersemiab} and corollary \ref{cohersemi}, respectively. 
\end{proof}

\begin{teo}
\begin{enumerate}
\item For $n \geq 1$, a simplicial group $X$ is torsion for $\mu_{n \geq}$ (i.e. it belongs to $Ker(Cot_n)$) if and only if and $tr_{n-1}(X)=0$ and $\eta_X$ is a normal epimorphism, where $\eta$ is the unit of $tr_n \dashv cosk_n$.	

\item For $n=0$, $X$ belongs to $Ker(Cot_0)$ if and only if $\eta_X$ is a normal epimorphism with $\eta$ the unit of $tr_0\dashv cosk_0$.
\end{enumerate} 
\end{teo}

\begin{proof} 1) Recall that $\mu_{n \geq} \leq \mu_{\geq n}$ we have the inclusion of torsion subcategories $Ker(Cot_n) \subseteq Ker(tr_{n-1})$. Then, a simplicial group $X$ belongs to $Ker(tr_{n-1})$ if and only if its Moore complex $M$ is trivial for degrees $n-1 \geq i$, and moreover $X$ belongs to $Ker(Cot_n)$ if and only if, in addition, $\delta_{n+1}: M_{n+1} \to M_n$ is a normal epimorphism.

Let $X$ belong to $Ker(tr_{n-1})$ and $\eta_X:X \to Cosk_n$ where $\eta$ is the unit of the adjunction $tr_n \dashv cosk_n$. The normalization of $\eta_X$ is
\[\begin{tikzcd} M(X)=\dots\ar[r]\ar[d, "M(\eta_X)"] & M_{n+2}\ar[r]\ar[d] & M_{n+1}\ar[r, "\delta_{n+1}"]\ar[d, "\delta_{n+1}"] & M_n\ar[r]\ar[d, "1"] & 0\ar[r]\ar[d]& \dots  \\ M(Cosk_n)=\dots\ar[r] & 0\ar[r] &M_n\ar[r, "1"] & M_n\ar[r] & 0\ar[r]& \dots      \end{tikzcd} \] 

Since the normalization functor is conservative we have that $\delta_{n+1}$ is a normal epimorphism if and only if $\eta_X$ is a normal epimorphism. 2) It is similar to 1).
\end{proof}

Notice, as expected from the torsion theory $(Conn(Grpd), Dis(Grp))$ in $Grpd(Grp)$, that the torsion category $Ker(Cot_0)$ in $Simp(Grp)$ contains the subcategory of connected internal groupoids $Conn(Grpd)$, i.e. internal groupoids $X$ with the condition that $(d_0,d_1): X_1 \to X_0^2$ is a normal epimorphism. 

\section{Homotopy groups and torsion subobjects}

\begin{defi}\label{funsimpgrp}  For $m\geq n$ and the idempotent radicals of the torsion theories of $\mu(Grp)$:
\[  \dots\leq \quad \mu_{m \geq} \quad \leq \quad \mu_{ \geq m} \quad \leq \dots \leq \quad \mu_{n \geq} \quad \leq \quad \mu_{\geq n} \quad \leq \dots ,\]
we consider the quotients of preradicals of $Simp(Grp)$:
\[
\Pi^{\geq n}_{m \geq}:=  \frac{\mu_{\geq n}}{\mu_{m \geq}}\, , \quad \Pi^{n\geq }_{ \geq m}:=  \frac{\mu_{n\geq }}{\mu_{ \geq m}} \, ,\quad \Pi^{\geq n}_{ \geq m}:=  \frac{\mu_{\geq n}}{\mu_{ \geq m}}\, , \quad \Pi^{n\geq }_{ m \geq }:=  \frac{\mu_{n\geq }}{\mu_{m \geq }}\, ; 
\]
as well as, for all $n$ the trivial quotients:
\[\Pi^{\geq n}:= \frac{\mu_{\geq n}}{0}\cong \mu_{\geq n}\, , \quad \Pi_{\geq n}:=\frac{Id}{\mu_{\geq n}}\, , \quad \Pi^{n \geq }:= \frac{\mu_{n \geq }}{0}\cong \mu_{n \geq }\, , \quad \Pi_{n \geq }:=\frac{Id}{\mu_{n \geq }} \, . \]
For a simplicial group $X$ the objects 
\[\Pi^{\geq n}_{m \geq}(X), \quad \Pi^{n\geq }_{ \geq m}(X), \quad \Pi^{\geq n}_{ \geq m}(X), \quad \Pi^{n\geq }_{ m \geq }(X), \quad \Pi^{\geq n}(X), \quad \Pi_{\geq n}(X), \quad \Pi^{n \geq }(X), \quad \Pi_{n \geq }(X)\] will be called the \textit{fundamental simplicial groups} of $X$. Accordingly, the family of functors:
 \[\Pi^{\geq n}_{m \geq}, \Pi^{n\geq }_{ \geq m}, \Pi^{\geq n}_{ \geq m}, \Pi^{n\geq }_{ m \geq }, \Pi^{\geq n}, \Pi_{\geq n}, \Pi^{n \geq }, \Pi_{n \geq }:\begin{tikzcd} Simp(Grp)\ar[r] & Simp(Grp) \end{tikzcd} \]
will be called \textit{fundamental simplicial functors}.
\end{defi}

Following Proposition \ref{porter2}, the homotopy groups of $\Pi_{n \geq}(X)=Id/\mu_{n \geq}(X)=Cot_n(X)$ are the same as $X$ for $ n \geq i$ and trivial elsewhere. The homotopy groups of the fundamental simplicial groups are the same as $X$ or trivial at some degrees. The following result generalizes 3) and 4) of \ref{porter2}.

\begin{teo}\label{funsimgrp}  Let be $X$ a simplicial group with Moore complex $M$. The homotopy groups of the fundamental simplicial group of $X$ are calculated as follows:
\begin{enumerate}
\item For all $n\geq 0$
\[\pi_i(\Pi^{n \geq}(X))=\pi_i(\Pi^{\geq n+1})= \left\{\begin{array}{cc}
  \pi_i(M) & i \geq n+1 \\
 0     & n+1 >i \, .
 \end{array}\right. \]
\item For all $n \geq 0$
\[\pi_i (\Pi_{n \geq}(X))= \pi_i(\Pi^{\geq n+1}(X))= \left\{\begin{array}{cc}
  0 & i \geq n+1 \\
 \pi_i(X)     & n+1 >i \, .
 \end{array}\right.\]
\item For all $n \geq 0$ 
\[ \pi_i(\Pi^{n \geq}_{\geq n+1}(X)) = 0 \ \mbox{for all} \ i \, .\]
\item For $m>n\geq 0$
\[ \pi_i (\Pi^{n \geq}_{\geq m+1}(X)) = \left\{\begin{array}{cc}
  \pi_i(X) &  m+1 > i \geq n+1 \\
 0  & \mbox{otherwise} \, .
 \end{array}\right.\]
\item Moreover, for $m>n \geq 0$ and for all $i$ 
\[\pi_i (\Pi^{n \geq}_{\geq m+1}(X)) = \pi_i(\Pi^{n \geq }_{m \geq}(X)) =\pi_i(\Pi^{\geq n+1}_{ \geq m+1}(X)) = \pi_i(\Pi^{\geq n+1}_{m \geq}(X))\, .\]
\item In particular, for $m=n+1$
\[ \pi_i (\Pi^{n \geq}_{\geq n+2} (X))= \left\{\begin{array}{cc}
  \pi_i(X) & i=n+1 \\
 0  & i\neq n \, .
 \end{array}\right.\]
\end{enumerate}  
\end{teo}  

\begin{proof} Since the Moore Normalization preserves short exact sequences and the preradicals $\mu_{n \geq}$ and $\mu_{\geq n }$ are mapped into the preradicals $cok_n$ and $ker_n$, this follows from the calculations of \ref{homology}.
\end{proof}

\begin{strc}\label{eilmac} For an abelian group $A$, a simplicial group $X$ is an Eilenberg-Mac Lane simplicial group of type $K(A,n)$ or a $K(A,n)$-simplicial group, if it has $\pi_n(X)=A$ and all other homotopy groups trivial. 

In particular, the $n$-th Eilenberg-Mac Lane simplicial group $K(A,n)$ for an abelian group $A$ (in symmetric form) is defined as follows. Consider the $(n+1)$-truncated simplicial group $k(A,n)$:
\[k(A,n)= \begin{tikzcd}  A^{n+1} \ar[r, shift left=3, "d_{n+1}"]\ar[r, shift right=3, "d_0"']\ar[r, phantom, "\vdots", shift left] & A \ar[r, shift left=3, "0"]\ar[r, shift right=3, "0"']\ar[r, phantom, "\vdots", shift left] & 0 \ar[r, shift left=3, "0"]\ar[r, shift right=3, "0"']& \dots \ar[r, shift left=3, "0"]\ar[r, shift right=3, "0"']& 0 \end{tikzcd}   \]
where the non-trivial face morphisms are 
\[(d_0, d_1, \dots, d_{n+1})=(p_0, p_0+p_1, p_1+p_2, \dots, p_{n-1}+p_n , p_n)\, , \]
where $p_i$ are the product projections and the degeneracies are given by $s_i=(0, \dots, 1_A, \dots, 0)$ with $1_A$ in the $i$th-place for $0 \leq i \leq n$. Then, we define 
\[K(A,n)=cosk_{n+1}(k(A,n)).\]
 It can be observed that for $m \geq n+1$, $K(A,n)_m=A^{\left( \frac{p}{n} \right)}$ where $\left( \frac{p}{n} \right)$ is the binomial coefficient.

Indeed, it is easy to see that the the Moore complex of $K(A,n)$ is:
\begin{equation}\label{K(A,n)}
M(K(A,n))= \begin{tikzcd}  \dots\ar[r] & 0\ar[r] & A\ar[r] & 0\ar[r] & \dots\ar[r] & 0\end{tikzcd}.
\end{equation}

This construction yields an embedding of the category $Ab$ of abelian groups into the category of simplicial groups $Simp(Grp)$ at any degree $n\geq 1$:
\[K(\_ ,n): \begin{tikzcd} Ab \ar[r]  & Simp(Grp)  \end{tikzcd}.\]
At $n=1$, it correspond with the usual definition of an abelian group as a simplicial group. For $n=0$ we will also consider the embedding of discrete simplicial groups $Dis: Grp \to Simp(Grp)$.
\end{strc}

\begin{strc}
The Dold-Kan Theorem  gives an equivalence between the categories of simplicial abelian groups $Simp(Ab)$ and chain complexes in abelian groups $chn(Ab)_{\geq 0}$, where the equivalence is given by the Moore normalization. In \cite{CaCe91} this equivalence was further extended to an equivalence between $Simp(Grp)$ and the category of hypercrossed modules in $Grp$. An hypercrossed module is a group chain complex $M$ with group actions for all $n$:
\[\Phi_\alpha^n: \begin{tikzcd} M_{r(\alpha)}\ar[r] & Aut(M_n) \end{tikzcd}\, \mbox{for}\, \alpha \in S(n)  \]
and binary operations
\[\Gamma^n_{\alpha, \sigma}: \begin{tikzcd} M_{r(\alpha)} \times M_{r(\sigma)}\ar[r] & M_n \end{tikzcd}\, \mbox{for}\, \alpha,\sigma \in S(n), \, 1< \sigma < \alpha, \, \alpha \cap \sigma= \emptyset \]
satisfying some equations. $S(n)$ has the order introduced in \ref{ordersn} and $\alpha \cap \sigma= \emptyset$ means that the maps  $\alpha, \sigma$ do not share a common index in their factorization by the degeneracies $\bar{i}$.
\end{strc}

\begin{lemma}\label{coroeilmac} Let $X$ be a simplicial group with Moore complex
\[M=\begin{tikzcd} \dots\ar[r] & 0\ar[r] & 0\ar[r] & A\ar[r] & 0\ar[r] & 0\ar[r] & \dots \end{tikzcd}\]
 then $X$ isomorphic to the Eilenbeg-Mac Lane simplicial group $K(A,n)$.
\end{lemma}
\begin{proof} Consider the chain complex $M$ as above. Since all degrees of $M$ are trivial except one, any morphism $M_i \to Aut(M_j)$ with $j>i$ is trivial as well as any binary mappings $M_i \times M_j \to M_k$ with $k>i,j$. This means that the structure of hypercrossed module is necessarily unique. Thus, it follows from the equivalence of hypercrossed modules and simplicial groups (see \cite{CaCe91}) that $X \cong K(A,n)$. 
\end{proof}

The next corollary generalizes part 4) of Proposition \ref{porter2}.

\begin{coro}\label{kanEL} For $n \geq 0$ and  a simplicial group $X$ the simplicial groups:
\[\Pi^{n \geq}_{\geq n+2}(X)\, ,\quad \Pi^{n \geq}_{n+1 \geq }(X)\, , \quad \Pi^{ \geq n+1}_{\geq n+2}(X)\, , \quad \Pi^{ \geq n+1}_{n+1 \geq }(X)\]
are $K(\pi_{n+1}(X), n+1)$-simplicial groups.

Moreover, $\Pi^{\geq n+1}_{n+1\geq}(X)$ is isomorphic to $K(\pi_{n+1}(X),n+1)$ the $(n+1)$-th Eilenberg-Mac Lane simplicial group of $\pi_{n+1}(X)$.
\end{coro}

\begin{proof} It follows by definition from 6) of theorem \ref{funsimgrp}.

 Moreover, the Moore complex of $\Pi^{\geq n+1}_{n+1 \geq}(X)$ isomorphic to the chain complex $\frac{ker_{n+1}(M)}{cok_{n+2}(M)}$, i.e.: 
 \[\begin{tikzcd} \dots\ar[r] & 0\ar[r] & 0\ar[r] & \pi_{n+1}(X) \ar[r] & 0\ar[r] & 0\ar[r] & \dots \end{tikzcd}\]
  It follows from lemma \ref{coroeilmac} that $\Pi^{\geq n+1}_{n++1 \geq}(X) \cong K(\pi_{n+1}(X) ,n+1)$.
\end{proof}

\begin{coro}
\begin{enumerate}
\item The fundamental functor $\Pi_{0 \geq}$ is naturally isomorphic to the connected component functor $\pi_0$ followed by the discrete functor:
\[\begin{tikzcd}  Simp(Grp)\ar[rr, "\Pi_{0 \geq}"]\ar[rd, "\pi_0"'] & & Simp(Grp) \\ & Grp\ar[ru, "Dis"'] & \end{tikzcd} \] 
\item For $n \geq 1$, the fundamental functor $\Pi^{\geq n+1}_{n+1 \geq}$ is naturally isomorphic to the homotopy group functor $\pi_n$ followed by the embedding $K(\_, n)$:
\[\begin{tikzcd}  Simp(Grp)\ar[rr, "\Pi_{n+1 \geq}^{\geq n+1}"]\ar[rd, "\pi_n"'] & & Simp(Grp) \\ & Ab \ar[ru, "{K(\_ ,n)}"'] & \end{tikzcd} \] 
\end{enumerate}
\end{coro}

\begin{proof} 1) From Corollary \ref{disgrpdtf}, the torsion theory $\mu_{0 \geq}$ has as  torsion-free reflector the functor $\Pi_{0\geq}=Cot_0= Dis \pi_0$ since $\pi_0(X) = coeq(d_0,d_1)=X_0/\delta_1(M_1)$. 2) It follows from Corollary \ref{kanEL}.  
\end{proof}

Following \cite{GabZis67}, the \textit{fundamental groupoid} or Poincar\'e groupoid $\Pi_1(X)$ of a simplicial set $X$ has as objects the set $X_0$, the vertices of $X$, and morphisms are generated by the elements of $X_1$ and their formal inverses and the relations $s_0(x)=1_x$ if $x\in X_0$ and $(d_0\sigma)(d_2\sigma)=d_1\sigma$ if $\sigma \in X_2$. Recently, in \cite{Duv19} the fundamental groupoid $\Pi_1: Simp(\X) \to Grpd(\X)$ has been studied for simplicial objects in an exact Mal'tsev category $\X$ as the left adjoint of the nerve functor $\mathcal{N}: Grpd(\X) \to Simp(\X)$. Indeed, if $\X$ is semi-abelian (in particular the category of groups as in our case) for a simplicial object $X$, $\Pi_1(X)$  is the unique groupoid  strucutre that has 
\[\begin{tikzcd} X_1/(d_2(Ker(d_0)\cap Ker(d_1)) )  \ar[r, shift right=2]\ar[r, shift left=2] & X_0\ar[l]  \end{tikzcd}\]
as the underlying reflexive graph, which for simplicial groups corresponds to the cotruncation functor $Cot_1$. Thus we have:

\begin{coro}
 The fundamental functor $\Pi_{1 \geq}$ is naturally isomorphic to the fundamental groupoid functor as indicated in the diagram:
\[\begin{tikzcd}  Simp(Grp)\ar[rr, "\Pi_{1 \geq}"]\ar[rd, "\Pi_1"'] & & \mathcal{M}_{1 \geq} \\ & Grpd(Grp) \ar[ru, "\mathcal{N}"'] & \end{tikzcd} \] 
\end{coro}

\section*{Acknowledgements}
This work is part of the author's Ph.D. thesis \cite{Lop22}. The author would like to thank his advisor Marino Gran for his guidance and suggestions.

{\footnotesize INSTITUT DE RECHERCHE EN MATH\'EMATIQUE ET PHYSIQUE,UNIVERSIT\'E CATHOLIQUE DE LOUVAIN, CHEMIN DU CYCLOTRON 2, 1348 LOUVAIN-LA-NUEVE, BELGIUM
\newline \textit{E-mail address}: guillermo.lopez@uclouvain.be}


\begin{thebibliography}{20}
\bibitem{Bar71} M. Barr, \textit{Exact categories: in Exact categories and categories of sheaves}, Springer Lec. Notes in Math \textbf{236}, 1971, 1-120

\bibitem{BorBou04} F. Borceux and D. Bourn, \textit{Mal'cev, protomodular, homological and semi-abelian categories}, Kluwer Academics Publishers, Dordrechts 2004

\bibitem{Bou91} D. Bourn, \textit{Normalization equivalence, kernel equivalence and affine categories}, Springer Lec. Notes in Math. \textbf{1488}, 1991, 43-62

\bibitem{BouGrn06} D. Bourn and M. Gran, \textit{Torsion theories in homological categories}, J. Algebra \textbf{305}, 2006, 18-47

\bibitem{BrHiSi10} R. Brown,  P.J. Higgins and R. Sivera, \textit{Nonabelian Algebraic Topology: Filtered spaces, crossed complexes, cubical homotopy groupoids}, EMS Tracts in Mathematics, 2010

\bibitem{CaCe91} P. Carrasco and A.M. Cegarra, \textit{Group-theoretic algebraic models for homotopy types}, J. Pure Appl. Algebra \textbf{75}, 1991, 195-235

\bibitem{CDT06} M. M. Clementino and D Dikranjan and W. Tholen, \textit{Torsion theories and radicals in normal categories}, J. Algebra \textbf{305}, 2006, 98-129

\bibitem{CGJ18} M. M. Clementino and M. Gran and G. Janelidze, \textit{Some remarks on Protolocalizations and protoadditive reflections}, Journal of Algebra and its Applications \textbf{17}(11), 2015, 105-124

\bibitem{Con84} D. Conduch\'e, \textit{Modules crois\'es g\'en\'eralis\'es de longeur 2}, J. Pure Appl. Algebra \textbf{34}, 1984, 155-178 

\bibitem{Duv19} A. Duvieusart, \textit{Fundamental groupoids for simplicial objects in Mal'cev categories}, J. Pure Appl. Algebra \textbf{226}(6), 2021, 106620

\bibitem{EveGrn13} T. Everaert and M. Gran, \textit{Monotone-light factorisation systems and torsion theories}, Bull. Sci. Math \textbf{23}(11), 2013, 221-242

\bibitem{EveLin04} T. Everaert and T. Van der Linden, \textit{Baer invariants  in semi-abelian categoreis ii: homology}, Th. Appl. Categ. \textbf{12}, 2004, 195-224

\bibitem{GabZis67} P. Gabriel and M. Zisman, \textit{Calculus of Fractions and Homotopy Theory}, Springer-Verlang, New York Heidelberg Berlin, 1967

\bibitem{GrnRss07} M. Gran and V. Rossi, \textit{Torsion theories and Galois coverings of topological groups}, J. Pure Appl. Algebra \textbf{208}, 2007, 135-151

\bibitem{ill71} L. Illusie, \textit{Complexe Cotangent et Deformations i}, Lecture Notes in Mathematics 239, Springer, 1971

\bibitem{JMT02} G. Janelizde and L. Marki and W. Tholen, \textit{Semi-Abelian Categories}, J. Pure Appl. Algebra \textbf{168}, 2002, 367-386

\bibitem{JMTU10} G. Janelidze and L. Marki and W. Tholen and A. Ursini, \textit{Ideal determined categories}, Cah. Top. G\'eom. Diff. Cat\'eg \textbf{51}(2), 2010, 115-125

\bibitem{JT07} G. Janelidze and W. Tholen, \textit{Characterization of torsion theories in general categories}, Contemp. Math. \textbf{431}, 2007, 249-256

\bibitem{zjane} Z. Janelidze, \textit{The pointed subobject functor, 3x3 lemmas, and substractivity of spans}, Theory Appl. Categ. \textbf{23}, 2010, 221-242

\bibitem{Lod82} J. Loday, \textit{Spaces with finitely many non-trivial homotopy groups}, J. Pure Appl. Algebra, \textbf{24}, 1982, 179-202

\bibitem{Lop22} G. A. L\'opez Cafaggi, \textit{Torsion Theories in Simplicial Groups:
preradicals and homology}, Ph.D. Thesis, Universit\'e Catholique de Louvain, Belgium, 2022.

\bibitem{May67} P. May, \textit{Simplicial objects in algebraic topology}, University of Chicago Press, 1967

\bibitem{Por93} T. Porter, \textit{N-types of simplicial groups and crossed N-cubes}, Topology \textbf{32}, 1993, 5-24

\bibitem{Sten} B. Stenstrom, \textit{Rings of quotients}, Springer-Verlang, New York Heidelberg Berlin, 1975

\end{thebibliography}
\end{document}